\documentclass[11pt]{article}
\usepackage{amsfonts}
\usepackage[utf8]{inputenc}

\usepackage{graphics}

\usepackage{indentfirst}
\usepackage{cite}
\usepackage{latexsym}
\usepackage{amsmath}
\usepackage{amssymb}
\usepackage[dvips]{epsfig}
\usepackage{amscd}
\def\on{\bar\rho}
\newtheorem{theorem}{Theorem}[section]
\newtheorem{remark}{Remark}[section]

\newtheorem{definition}{Definition}[section]
\newtheorem{lemma}[theorem]{Lemma}

\newtheorem{corollary}[theorem]{Corollary}
\newtheorem{proposition}[theorem]{Proposition}

\newcommand{\n}{\rho}

\newcommand{\ti}{\tilde}

\renewcommand{\div}{ {\rm div }  }

\newcommand{\pa}{\partial}
\renewcommand{\r}{\mathbb{R}}
\newcommand{\bi}{\bibitem}

\newcommand{\ia}{\int_0^T}

\newcommand{\bt}{\begin{theorem}}
\newcommand{\bl}{\begin{lemma}}
\newcommand{\el}{\end{lemma}}
\newcommand{\et}{\end{theorem}}
\newcommand{\ga}{\gamma}

\newcommand{\curl}{{\rm curl} }

\newcommand{\de}{\delta}
\newcommand{\ve}{\varepsilon}
\newcommand{\la}{\label}

\newcommand{\ol}{\overline}

\newcommand{\bn}{\begin{eqnarray}}
\newcommand{\en}{\end{eqnarray}}
\newcommand{\bnn}{\begin{eqnarray*}}
\newcommand{\enn}{\end{eqnarray*}}

\newcommand{\bnnn}{\begin{eqnarray*}}
\newcommand{\ennn}{\end{eqnarray*}}

\newcommand{\ba}{\begin{aligned}}
\newcommand{\ea}{\end{aligned}}
\newcommand{\be}{\begin{equation}}
\newcommand{\ee}{\end{equation}}
\def\O{{\Omega }}

\def\norm[#1]#2{\|#2\|_{#1}}

\newcommand{\no}{\nonumber\\}

\newcommand{\si}{\sigma}

\def\la{\label}

\def\na{\nabla}
\def\on{\bar\n}

\makeatletter      % '@' is now a normal "letter" for TeX
\@addtoreset{equation}{section}
\makeatother

\makeatletter      % '@' is now a normal "letter" for TeX
\@addtoreset{equation}{section}
\makeatother       % '@' is restored as a "non-letter" character for TeX

\title{Global Classical Solutions to the Compressible Navier-Stokes Equations with Slip Boundary Conditions in 3D Exterior Domains }

 \author{Guocai C{\small AI}$^{a}$,
   Jing L{\small I}$^{ b,c,d} $, Boqiang L{\small \"U}$^{e} $ \thanks{email:  gotry@xmu.edu.cn (G. C. Cai), ajingli@gmail.com (J. Li), lvbq86@163.com (B. Q.  L\"u) }  \\
{\normalsize a.  School of Mathematical Sciences, }\\ {\normalsize  Xiamen University, Xiamen 361005, P. R. China;}\\
{\normalsize b. Department of Mathematics, }\\ {\normalsize  Nanchang University, Nanchang 330031, P. R. China;} \\ {\normalsize c. Institute of Applied Mathematics, AMSS,} \\ {\normalsize \&   Hua Loo-Keng Key Laboratory of Mathematics,}\\
{\normalsize  Chinese Academy of Sciences,    Beijing 100190,
P. R. China;}
 \\ {\normalsize d.  School of Mathematical Sciences,}\\
{\normalsize  University of Chinese Academy of Sciences, Beijing 100049, P. R. China}
\\ {\normalsize e.  College of Mathematics and Information Science,}\\
{\normalsize   Nanchang Hangkong University, Nanchang 330063, P. R. China}}

\date{ }

\begin{document}
\maketitle

 \begin{abstract} We are concerned with the global existence of classical solutions
 to the barotropic compressible Navier-Stokes equations with slip boundary condition in a three-dimensional (3D) exterior domain.  We demonstrate that the classical solutions exists globally in time provided that the initial total energy is suitably small. It is worth noting that the initial density is allowed to have large oscillations and contain vacuum states. For our purpose, some new techniques and methods are adopted to obtain necessary a priori estimates, especially the estimates on the boundary. Moveover, we also give the large-time behavior of the classical solutions what we have gotten.
 \end{abstract}


\section{Introduction}

The viscous barotropic compressible Navier-Stokes for isentropic flows reveal the principles of conservation of mass and momentum in the absence of exterior forces:
   \be \la{a1}  \begin{cases}\n_t+{\rm div} (\n u)=0,\\
 (\n u)_t+{\rm div}(\n u\otimes u)-\mu\Delta u-(\mu+\lambda)\na {\rm
div} u +\na P(\n) =0, \end{cases}\ee
 where $(x,t)\in\Omega\times (0,T]$, $\Omega$ is a domain in $\r^{3}$, $t\ge 0$ is time,   $x$  is the spatial coordinate. $\rho\geq0, u=(u^1,u^2,u^3)$ and $P(\rho)=a\rho^{\gamma}$$(a>0,\gamma>1)$
are the unknown fluid density, velocity and pressure, respectively. The constants $\mu$ and $\lambda$ are the shear viscosity and bulk coefficients respectively  satisfying the following physical restrictions: \be\la{h3} \mu>0,\quad 2\mu +  {N} \lambda\ge 0.
\ee

This paper can be regarded as  a continuation of our previous work \cite{CCL1}. Unlike the previous article, we study the global existence of classical solutions of \eqref{a1} in the exterior of a simply connected bounded domain in $\r^3$ rather than in the inner bounded domain. More precisely, three-dimensional  problem of the system \eqref{a1} will be investigated under the assumptions: The domain $\Omega$ is the exterior of a simply connected bounded domain $D$ in $\r^3$, i.e. $\Omega=\r^3-\bar{D}$ and its boundary $\partial\Omega$ is smooth. In addition, the system \eqref{a1} is discussed subject to the given initial data
\be \la{h2} \n(x,0)=\n_0(x), \quad \n u(x,0)=\n_0u_0(x),\quad x\in \Omega,\ee
and slip boundary condition
\be \la{ch1} u\cdot n=0,\,\,\,\,\curl u\times n=0 \,\,\,\text{on} \,\,\,\partial\Omega,\ee
with the far field behavior
\be \la{ch2} u(x,t)\rightarrow0,\,\,\rho(x,t)\rightarrow\rho_\infty\geq0,\,\,as\,\, |x|\rightarrow\infty, \ee
where $n=(n^1,n^2,n^{3})$ is the unit normal vector to the boundary $\partial\Omega$ pointing outside $\Omega$, $\rho_\infty$ is a given constant. Our purpose is to show that such  global classical solutions indeed exists provided that the initial energy is sufficiently small. It is worth noting that in this paper the initial density is allowed to have possibly large oscillations with constant state as far field which could be either vacuum or non-vacuum. Since $\Omega$ is no longer bounded, there are some distinguish differences from the previous work \cite{CCL1}.

There are a considerable number of literatures on the large time existence and
behavior of solutions to (\ref{a1}). The one-dimensional problem
has been studied extensively, see
\cite{Kaz,Ser1,Ser2,Hof} and the references therein. For the
multi-dimensional case, the local existence and uniqueness of
classical solutions are known in \cite{Na,se1}  in the absence of
vacuum and recently, for strong solutions also, in \cite{K3,K1,
K2, S2,hlma} for the case that the initial density need not be positive
and may vanish in open sets. The global classical solutions were
first obtained by Matsumura-Nishida \cite{M1} for initial data
close to a non-vacuum equilibrium in $H^3.$ In
particular, the theory requires that the solution has small
oscillations from a uniform non-vacuum state so that the density
is strictly away from vacuum. Later, Hoff \cite{H3,Hof2,Ho3}
studied the problem for discontinuous initial data. For the
existence of solutions for arbitrary data, the major breakthrough is due to Lions \cite{L1} (see also  Feireisl  \cite{F1,F2}), where the
global existence of weak solutions  when the exponent $\ga$ is suitably large are achieved. The
main restriction on initial data is that the initial total energy is
finite, so that the density vanishes at far fields, or even has
compact support. However, little is known on the structure of such
weak solutions, particularly, the regularity and the uniqueness of such weak solutions remain open.
%Recently, motivated by Huang-Li-Xin  \cite{hlx1} (see also \cite{jx01}), Cai-Li \cite{CCL1} proved that the classical solution of the problem \eqref{a1} with slip boundary condition in a 3D simply connected bounded domain exists globally in time with vacuum and small energy but possibly large oscillations. the global existence and uniqueness of classical solutions with constant state as far field which could be either vacuum or non-vacuum to the Cauchy problem \eqref{a1} in three-dimensional space with smooth initial data which are of small total energy  but possibly large oscillations.
Here we refer the reader to the books by Lions \cite{L1}, Feireisl \cite{F2}, Novotn\'{y} $\&$ Stra$\check{s}$kraba \cite{ANIS} and Giga $\&$ Novotn\'{y} \cite{ygan} for further details.

The choice of the boundary conditions is more complex when one discusses the system \eqref{a1} in a region with boundaries. One of the frequent choices is the homogeneous Dirichlet boundary condition $u = 0$ on $\partial\Omega$ in the case when $\partial\Omega$ represents a fixed wall. This condition was formulated by G. Stokes in 1845, which is also called no-slip boundary condition expressing the physical fact that a real fluid adheres to $\partial\Omega$. However, In the case where the obstacles have a rough boundary, the no-slip boundary condition is no longer valid (see for instance \cite{se2}). An alternative was suggested by H. Navier even before (in 1824), who proposed the conditions as follows:
\be \la{Navi}
u \cdot n = 0, \,\,(2D(u)\,n+ \vartheta u)_{tan}=0 \,\,\,\text{on}\,\,\, \partial\Omega,
\ee
%=(D(u)\,n+ \vartheta u)-((D(u)\,n+ \vartheta u)\cdot n)\,n
where $D(u) = (\nabla u+(\nabla u)^{\rm tr})/2$ is the shear stress, $\vartheta$ is a scalar friction function, and the symbol $v_{tan}$ represents the projection of tangent plane of the vector $v$ on $\partial\Omega$. We also call it Navier-type slip condition  in which there is a stagnant layer of fluid close to the wall allowing a fluid to slip and the slip velocity is proportional to the shear stress.
Such boundary conditions can be induced by effects of free capillary boundaries (see \cite{BE1}), or a rough boundary as in \cite{Apv1, Jwm1}, or a perforated boundary, which is then called Beavers-Joseph's law, see \cite{Scg1, Bgs1}, or an exterior electric field as in \cite{Cde1}. The Navier-type slip condition has been employed in a wide range of problem, including free surface problems (see e.g.,\cite{Sva1}), turbulence modeling (see e.g.,\cite{Glw1}) and inviscid limits (see e.g.,\cite{Xyxz1, dvhb1}). It is a remarkable fact that when $\partial\Omega$ is of constant curvature, \eqref{ch1} is equivalent to \eqref{Navi} with $\vartheta=\kappa$ (see \cite{CCL1}),
where $\kappa$ is the corresponding principal curvature of $\partial\Omega$. One can see \cite{Sva1, Mpb1, Xyxz1, dvhb1, Vka1, Ho3} for the studies of Navier-Stokes equations with Navier-type slip boundary condition. However, as far as we know, there are few research works on the compressible Navier-Stokes equations with slip boundary condition in an unbounded domain, except for the results of \cite {Ho3} and \cite {F3}, which give the global existence of weak solutions in a half-space and the weak-strong uniqueness property in the class of finite energy weak solutions in an unbounded domain respectively.

The outline of this paper is as follows. First, we will give our main results at the end of  this section. In Section 2,  some notations, known facts and elementary inequalities needed in later analysis are prepared for our discussion. Section 3 and Section 4 are devoted to deriving the necessary a priori estimates on classical solutions which can guarantee the local classical solution to be a global classical one. Finally, we will prove the main results, Theorems \ref{th1} and \ref{th2} in Section 5.

Before stating the main results, we explain some functional spaces, notations and conventions used throughout this paper. For a positive integer $k$ and $1\leq q<+\infty$, the standard homogeneous  Sobolev spaces are denoted as follows:
$$ D^{k,q}(\Omega)=\left\{u\in L_{loc}^{1}(\Omega)\left|\|\nabla^k u\|_{L^q(\Omega)}<+\infty\right.\right\},\,\,~~\|\nabla u\|_{D^{k,q}(\Omega)}\triangleq\|\nabla^k u\|_{L^q(\Omega)};$$
$$W^{k,q}(\Omega)=L^q(\Omega)\cap D^{k,q}(\Omega),\,\,\text{with the norm}~ \| u\|_{W^{k,q}(\Omega)}\triangleq\left(\sum\limits_{|m|\leq k} \|\nabla^m u\|^q_{L^q(\Omega)}\right)^{\frac{1}{q}};$$
\quad\,\, $D^k(\Omega)=D^{k,2}(\Omega)$, $H^k(\Omega)=W^{k,2}(\Omega)$.

For simplicity, we denote $L^q(\Omega)$, $D^{k,q}(\Omega)$, $D^k(\Omega)$, $W^{k,q}(\Omega)$ and $H^k(\Omega)$  by $L^q$, $D^{k,q}$, $D^k$, $W^{k,q}$ and $H^k$  respectively, and set
$$B_R\triangleq\{x\in\r^3||x|<R\},\,\,\,\int fdx\triangleq\int_\Omega fdx. $$

%When $X$ is a Banach space, for $q\in[1,\infty]$ , $$L^{q}(0,T;X)=\{u~\text{:}~u(t)\in X ~\text{for any} ~t\in (0,T), \text{and}~ \|u(t)\|_X\in L^{q}(0,T) \}$$  with  norm $\|u\|_{L^{q}(0,T;X)}\triangleq(\int_0^{T}\|u(t)\|^{q}_Xdt)^{\frac{1}{q}}$.

For two $3\times 3$  matrices $A=\{a_{ij}\},\,\,B=\{b_{ij}\}$,  the trace of $AB$ is presented by $A\colon B$,  that is,
$$ A\colon  B\triangleq \text{tr} (AB)=\sum\limits_{i,j=1}^{3}a_{ij}b_{ji}.$$

Finally,  we denote  $\nabla_iv=(\partial_iv^1,\partial_iv^2,\partial_iv^3),\,\,i=1,2,3,$ and the
material derivative of $v$   by  $\dot v\triangleq v_t+u\cdot\nabla v$, when $v=(v^1,v^2,v^3).$ Furthermore, $v\cdot \nabla u\triangleq(v\cdot \nabla u^1, v\cdot \nabla u^2, v\cdot \nabla u^3)$ and $\nabla u\cdot v\triangleq (\nabla_1 u\cdot v, \nabla_2 u\cdot v, \nabla_3 u\cdot v)$.

 Let $(\rho_0,u_0)$ and $\rho_\infty$ be given by \eqref{h2} and \eqref{ch2} respectively. The initial total energy of (\ref{a1}) is defined by
\be \la{c0}
C_0 =\int_{\Omega}\left(\frac{1}{2}\n_0|u_0|^2 + G(\rho_0) \right)dx,
\ee
where
\bnn
G(\rho)\triangleq\rho\int_{\rho_\infty}^{\rho}\frac{P(s)-P(\rho_\infty)}{s^{2}} ds.
\enn
%We consider first
%the two-dimensional case, that is, $\Omega\subset\r^{2}.  $ Without loss of generality, assume that the initial density $\n_0$ satisfies
%\be\la{oy3.7} \int_{\Omega} \n_0dx=1,\ee  which implies that there exists a positive constant $N_0$ such that  \be\la{oy3.8} \int_{B_{N_0}}  \n_0  dx\ge %\frac12\int\n_0dx=\frac12.\ee

Now we give our first result, which indicates that the existence and large-time behavior of global classical solutions to the problem  \eqref{a1}-\eqref{ch2}.
\begin{theorem}\la{th1} Let $\Omega$ be the exterior of a simply connected bounded domain $D$ in $\r^3$ and its boundary $\partial\Omega$ is smooth. For some $q\in(3,6)$  and two given constants $M$, $\bar{\rho}\geq \rho_\infty+1$, assume that  the initial data $(\n_0,u_0)$ is given as follows:
\be\ba \la{dt1}  & u_0\in\left\{f\in
 D^1\cap D^2:f\cdot n=0, ~~ {\rm curl}f\times n=0 ~\mbox{ \rm on } ~\partial \Omega\right\} ,\\ &\quad (\rho_0-\rho_\infty,P(\rho_0)-P(\rho_\infty))\in  H^2\cap W^{2,q}
,  \ea\ee
\be\la{dt2} 0\leq\rho_0\leq\bar{\rho},\,\,\|\nabla u_0\|_{L^2}\leq M, \ee
$$\rho_0\in L^{3/2}\,\,\, \text{if}\,\,\, \rho_\infty=0,$$
and the compatibility condition
\be\la{dt3}
-\mu\triangle u_0-(\mu+\lambda)\nabla\div u_0 + \nabla P(\rho_0) = \rho_0^{\frac{1}{2}}g, \ee
for some  $ g\in L^2.$  Then
there exists a positive constant $\ve$ depending only on  $\mu ,  \lambda ,   \ga ,  a ,    \on$,  $\Omega$ and $M$  such that if $C_0\le\ve$, then
the system \eqref{a1}-\eqref{ch2} has a unique global classical solution $(\n,u)$ in $\Omega\times(0,\infty)$ satisfying
\be\la{dt5}
  0\le\n(x,t)\le 2\bar{\n},\quad  (x,t)\in \O\times(0,\infty),
\ee
\be\la{dt6}\begin{cases}
(\rho-\rho_\infty,P-P(\rho_\infty))\in C([0,\infty);H^2\cap W^{2,q} ),\\  \na u\in C([0,\infty);H^1 )\cap  L^\infty_{\rm loc}(0,\infty; W^{2,q}),\\
u_t\in L^{\infty}_{\rm loc}(0,\infty; D^1\cap D^2)\cap H^1_{\rm loc}(0,\infty; D^1),\\   \sqrt{\n}u_t\in L^\infty(0,\infty;L^2).
\end{cases}\ee
In addition, the following large-time behavior
\be  \la{qa1w} \lim_{t\rightarrow \infty}\left(\|\n-\rho_\infty\|_{L^r}+\|\rho^\frac{1}{8}u\|_{L^4}+\|\nabla u\|_{L^2} \right)=0,\ee
holds for all $r\in(2,\infty)$ if $\rho_\infty>0$ and $r\in(\gamma,\infty)$ if $\rho_\infty=0$.
\end{theorem}
The following theorem shows the large-time behavior of the gradient of the
density when vacuum states appear initially.
\begin{theorem}\la{th2}
Under the conditions of Theorem \ref{th1}, assume further  that $\rho_\infty>0$ and
there exists some point $x_0\in \Omega$ such that $\n_0(x_0)=0.$  Then the unique
global classical solution $(\n,u)$ to the   problem  \eqref{a1}-\eqref{ch2} obtained in
Theorem \ref{th1}  satisfies that for any $r>3,$
\be\la{qa2w}\ba \lim_{t\rightarrow \infty}\|\na\n (\cdot,t)\|_{L^r}=\infty. \ea\ee
\end{theorem}
\begin{remark} Since $q>3,$ it follows from Sobolev's inequality and \eqref{dt6}$_1$  that \be\la{soh1}  \n-\rho_\infty,\na \n \in C(\bar\Omega\times [0,T]).\ee
Moreover, it also follows from \eqref{dt6}$_2$ and \eqref{dt6}$_3$ that  \be \la{soh2} u,\na u, \na^2 u, u_t \in C(\bar\Omega\times [\tau,T]),\ee due to the following simple fact that $$L^2(\tau,T;H^1)\cap H^1(\tau,T;H^{-1})\hookrightarrow C([\tau,T];L^2).$$
Finally, by \eqref{a1}$_1,$ we have \bnn \n_t=-u\cdot\na \n-\n\div u\in C(\bar\Omega\times [\tau,T]),\enn which together with \eqref{soh1} and \eqref{soh2} shows that the solution obtained by Theorem \ref{th1} is a classical one.
\end{remark}
\begin{remark}\la{u0A}
One also can prove the same conclusions of Theorems \ref{th1} and \ref{th2} under the following more wide boundary condition:
\be \la{ch101} u \cdot n = 0, \,\,\curl u\times n=-A\, u \,\,\,\text{on}\,\,\, \partial\Omega,\ee
where $A=A(x)$ is $3\times 3$ symmetric matrix which has compact support in $\r^3$ and $A\in W^{2,6}$ satisfying the conditions of Remark 1.3 in \cite{CCL1}. Since the treatments of some additional priori estimates are parallel to those of \cite{CCL1}, we omit detailed proofs.
\end{remark}

  Due to the equivalence of norms $\|\nabla v\|_{L^2}\simeq \|D (v)\|_{L^2}$ for any $v\in D^1$ with $v \cdot n = 0$ and  $v\rightarrow0\,\,as\,\, |x|\rightarrow\infty$ (see \cite{DMR1}),  similar to the proof of \cite[Theorem 1.4]{CCL1}, we immediately have the following conclusion.
 \begin{theorem}\label{th3}  Set $A=B-2D(n)$ in \eqref{ch101}, where $B\in W^{2,6}(\Omega)$ is a positive semi-definite $3\times 3$ symmetric matrix. Under the same assumptions of Theorem \ref{th1} expect that
 the first condition of \eqref{dt1} is replaced by $$u_0\in\left\{f\in
 D^1\cap D^2:f\cdot n=0, ~~ {\rm curl}f\times n=-A\, u ~\mbox{ \rm on } ~\partial \Omega\right\}, $$
 the conclusions of Theorems \ref{th1}-\ref{th2} with respect to the slip boundary condition \eqref{ch101} (instead of \eqref{ch1}) remain valid provided that $2\mu +  3 \lambda>0$.
\end{theorem}
As a direct result of Theorem \ref{th3}, for compressible Navier-Stokes equations \eqref{a1} with Navier-type slip boundary condition \eqref{Navi}, we give the following conclusion on the global existence and large-time behavior of classical solutions.
\begin{corollary} Assume that $\vartheta\in W^{2,6}$, the shear viscosity coefficient $\mu$ and the bulk one $\lambda$ satisfy one of the following two conditions:

(1)  $A=\vartheta I-2D(n)$ satisfies the assumption given by Remark \ref{u0A}.

(2)  $2\mu+3\lambda> 0$, $\vartheta\geq 0$.

Then, for the system \eqref{a1}-\eqref{h2} with Navier-type slip boundary condition \eqref{Navi} and the far field behavior \eqref{ch2}, the conclusions of Theorems \ref{th1}-\ref{th3} hold.
\end{corollary}

We now comment on the analysis of this paper. Indeed, compared with the previous results (\cite{CCL1}) where we treated the problem in a simply connected bounded domain, since the domain is unbounded, there are two difficulties that we have to overcome: one of them is how to estimate $\nabla u$ by means of $\div u$ and $\curl u$, the other one is how to control the boundary integrals, especially (see \eqref{cax31}), $$\int_{\partial\Omega}\sigma^{m}(\nabla F\cdot u)(\dot{u}\cdot n)ds.$$
For the former, we establish some necessary inequalities in Section 2 related to $\nabla u$, $\div u$ and $\curl u$, see Lemmas \ref{crle3}-\ref{crle5}. For the latter, we introduce a smooth 'cut-off' function defined on a ball containing $\r^3-\Omega$, which not only enables us to eliminate the first derivative in the boundary integral above by using the method in \cite{CCL1} based on the divergence theorem and the fact $u=u^\perp\times n$ (see \eqref{bdd2}), but also reduces the boundary integral to the integral on the ball. The reader can see \eqref{cax31} for details. All these treatments are the key to achieve our purposes in this paper.
\section{Preliminaries}\la{se2}
\subsection{Some known inequalities and facts}
%дÈëÎÒÃÇÐèÒªÓõ½µÄǶÈ붨Àí

In this subsection, some facts and elementary inequalities, which will be used frequently later, are collected.

First, similar to the proof of \cite[Theorem 1.4]{hxd1}, we have the local existence of strong and classical solutions.
\begin{lemma}\la{loc1} Let $\Omega$ be as in Theorem \ref{th1}, assume that $(\n_0,u_0)$ satisfies \eqref{dt1} and \eqref{dt3}. Then there exist a small time $T>0$ and a unique strong solution $(\n,u)$ to the problem \eqref{a1}-\eqref{ch1} on $\Omega\times(0,T]$ satisfying for any $ \tau\in(0,T),$
\be\nonumber\begin{cases}
 (\rho-\rho_\infty,P-P(\rho_\infty))\in C[0,\infty);H^2\cap W^{2,q} ),\\ u\in C([0,\infty);D^1\cap D^2 ),\,\,\,\na u\in  L^2(0,T; H^2)\cap  L^{p_0}(0,T; W^{2,q}),\\
 \nabla u\in L^{\infty}(\tau,T;H^2\cap W^{2,q})\\
u_t\in L^{\infty}(\tau,T; D^1\cap D^2)\cap H^1(\tau,T; D^1),\\   \sqrt{\n}u_t\in L^\infty(0,T;L^2),
\end{cases}\ee
where $q\in(3,6)$ and $p_0=\frac{9q-6}{10q-12}\in(1,\frac{7}{6}).$
\end{lemma}
The following Lemma  can be found in \cite{fcpm}.
\begin{lemma}
[Gagliardo-Nirenberg]\la{l1} Assume that $\Omega$ is the exterior of a simply connected domain $D$ in $\r^3$. For  $p\in [2,6],\,q\in(1,\infty) $ and
$ r\in  (3,\infty),$ there exists some generic
 constant
$C>0$ which may depend  on $p,q $ and $r$ such that for  any $f\in H^1({\O }) $
and $g\in  L^q(\O )\cap D^{1,r}(\O), $
\be\la{g1}\|f\|_{L^p(\O)}\le C\|f\|_{L^2}^{\frac{6-p}{2p}}\|\na
f\|_{L^2}^{\frac{3p-6}{2p}},\ee
\be\la{g2}\|g\|_{C\left(\ol{\O }\right)} \le C
\|g\|_{L^q}^{q(r-3)/(3r+q(r-3))}\|\na g\|_{L^r}^{3r/(3r+q(r-3))}.
\ee
\end{lemma}
Generally, \eqref{g1} and \eqref{g2} are called Gagliardo-Nirenberg's inequalities.
%%ÕâÀïдÈë¼£¶¨Àí%%
%·ÖÊý½×µÄ¼£¶¨Àí
%(2) For $u\in W^{1,p}(\Omega)$, $ 1< p \leq \infty$, then $u\in \dot{W}^{1-\frac{1}{p},p}(\partial\Omega)$ and there exists a positive constant %$C=C(n,p,\Omega)$ such that
%$$\|u\|_{\dot{W}^{1-\frac{1}{p},p}(\partial\Omega)}\leq C\|u\|_{W^{1,p}(\Omega)},$$
%where $\dot{W}^{1-\frac{1}{p},p}(\partial\Omega)=\{u\in L^p(\Omega)~\text{:} %\int_{{\partial\Omega}\times{\partial\Omega}}\frac{|u(x)-u(y)|^p}{|x-y|^{n-2+p}}dx<+\infty\},\,\,\text{with the norm:}$
% $$\| u\|_{\dot{W}^{1-\frac{1}{p},p}(\partial\Omega)}\triangleq %\|u\|_{L^p(\partial\Omega)}+\left(\int_{{\partial\Omega}\times{\partial\Omega}}\frac{|u(x)-u(y)|^p}{|x-y|^{n-2+p}}dx\right)^{\frac{1}{p}}.$$
%\end{lemma}
%
%Èýά±¸ÓÃ
%\begin{lemma}\la{th01}
%Assume that $\Omega$ is a bounded domain in $\r^2$ with $C^3$ boundary, then for any $v\in H^1$ satisfying $v\cdot n=h$ on $\partial\Omega$, where $h\in %H^\frac{1}{2}(\partial\Omega)$, it holds that
%\be\la{g12}
%\|\nabla v\|_{L^2(\Omega)}\leq C(\|\div v\|_{L^{2}(\Omega)}+\|\curl v\|_{L^{2}(\Omega)}+\|h\|_{H^\frac{1}{2}(\partial\Omega)}),
%\ee
%where ${H^\frac{1}{2}(\partial\Omega)}=\dot{W}^{\frac{1}{2},2}(\partial\Omega)$.
%\end{lemma}

Next,  in order to arrive at the
uniform (in time) upper bound of the density $\n,$ we need the following Zlotnik  inequality \cite{zl1}.
\begin{lemma}\la{le1}   Assume that $ g\in C(R)$ and $y,b\in W^{1,1}(0,T)$, and the function $y$ satisfy
\bnn y'(t)= g(y)+b'(t) \mbox{  on  } [0,T] ,\quad y(0)=y^0, \enn
If $g(\infty)=-\infty$
and \be\la{a100} b(t_2) -b(t_1) \le N_0 +N_1(t_2-t_1)\ee for all
$0\le t_1<t_2\le T$
  with some $N_0\ge 0$ and $N_1\ge 0,$ then
\bnn y(t)\le \max\left\{y^0,\overline{\zeta} \right\}+N_0<\infty
\mbox{ on
 } [0,T],
\enn
  where $\overline{\zeta} $ is a constant such
that \be\la{a101} g(\zeta)\le -N_1 \quad\mbox{ for }\quad \zeta\ge \overline{\zeta}.\ee
\end{lemma}

Consider the Neumann boundary value problem
\be\la{cxtj1}\begin{cases}
-\Delta v=\div f\,\, &\text{in}~ \Omega, \\
\frac{\partial v}{\partial n}=-f\cdot n\,&\text{on}~\partial\Omega,\\
\nabla v\rightarrow0,\,\,as\,\,|x|\rightarrow\infty,
\end{cases} \ee
where $v=(v^{1},v^{2},v^{3})$ and $f=(f^{1},f^{2},f^{3})$. Indeed, the problem is equivalent to
$$\int\nabla v\cdot\nabla\eta dx=\int f\cdot\nabla\eta dx,\,\,\forall\eta\in C_0^{\infty}(\r^3).$$
Thanks to \cite[Lemma 5.6]{ANIS}, we have the following conclusion.
\begin{lemma}   \la{zhle}
For the system \eqref{cxtj1}, one has

(1) If $f\in L^q$ for some $q\in(1,\infty),$ then there exists a unique (modulo constants) solution $v\in D^{1,q}$ such that
$$\|\nabla v\|_{L^q}\leq C(q,\Omega)\|f\|_{L^q}.$$

(2) If $f\in W^{k,q}$ for some $q\in(1,\infty),\,k\geq 1$, then $\na v\in W^{k,q}$ and
$$\|\nabla v\|_{W^{k,q}}\leq C \|f\|_{W^{k,q}}.$$
\end{lemma}
\begin{definition}
Let $\Omega$ be a domain in $\r^3$. If the first Betti number of $\Omega$ vanishes, namely, any simple closed curve in $\Omega$ can be contracted to a point, we say that $\Omega$ is simply connected. If the second Betti number of $\Omega$ is zero, we say that $\Omega$ has no holes.
\end{definition}
The following two lemmas are given in \cite{vww,CANEHS}, more precisely, Theorem 3.2 in \cite{vww} and  Propositions 2.6-2.9 in\cite{CANEHS}.
%or $\Omega$ has no holes and $v\times n=0$ on $\partial\Omega$
%$$\|v\|_{W^{k+1,q}}\leq C(\|\div v\|_{W^{k,q}}+\|\curl v\|_{W^{k,q}}+\|v\|_{L^q}).$$
\begin{lemma}   \la{crle3}
Let $k\geq0$ be a integer, $1<q<+\infty$, and assume that $D$ is a simply connected bounded domain in $\r^3$ with $C^{k+1,1}$ boundary $\partial D$. Then for $v\in W^{k+1,q}(D)$ with $v\cdot n=0$ on $\partial D$, there exists a constant $C=C(q,k,D)$ such that
\be\|v\|_{W^{k+1,q}(D)}\leq C(\|\div v\|_{W^{k,q}(D)}+\|\curl v\|_{W^{k,q}(D)}).\ee
In particular, for $k=0$, we have
\be \la{paz1}\|\nabla v\|_{L^q(D)}\leq C(\|\div v\|_{L^q(D)}+\|\curl v\|_{L^q(D)}).\ee
\end{lemma}
\begin{lemma}   \la{crle4}
  Suppose that $D$ is a bounded domain in $\r^3$ and its $C^{k+1,1}$ boundary $\partial D$ only has a finite number of 2-dimensional connected components. Then for  an  integer  $k\geq0$  and $1<q<+\infty$, there exists a constant $C=C(q,k,D)$ such that every $v\in W^{{k+1,q}}(D)$ with $v\times n=0$ on $\partial D$ satisfies
$$\|v\|_{W^{k+1,q}(D)}\leq C(\|\div v\|_{W^{k,q}(D)}+\|\curl v\|_{W^{k,q}(D)}+\|v\|_{L^q(D)}).$$
In particular, if  $D$ has no holes, then
$$\|v\|_{W^{k+1,q}(D)}\leq C(\|\div v\|_{W^{k,q}(D)}+\|\curl v\|_{W^{k,q}(D)}).$$
\end{lemma}
The following conclusion is shown in \cite[Theorem 3.2]{vww}.
\begin{lemma}   \la{crle1}
Let $D$ be a simply connected domain in $\r^3$ with $C^{1,1}$ boundary, and $\Omega$ is the exterior of $D$. For $v\in D^{1,q}(\Omega)$ with $v\cdot n=0$ on $\partial\Omega$, it holds that
\be\la{ljq01}\|\nabla v\|_{L^{q}(\Omega)}\leq C(\|\div v\|_{L^{q}(\Omega)}+\|\curl v\|_{L^{q}(\Omega)})\,\,\,for\,\, any\,\, 1<q<3,\ee
and
$$\|\nabla v\|_{L^q(\Omega)}\leq C(\|\div v\|_{L^q(\Omega)}+\|\curl v\|_{L^q(\Omega)}+\|\nabla v\|_{L^2(\Omega)})\,\,\,for\,\, any\,\, 3\leq q<+\infty.$$
\end{lemma}

Due to \cite[Theorem 5.1]{lhm} (choosing $\alpha=0$), we obtain
\begin{lemma}   \la{crle2}
Let $\Omega$ be given in Lemma \ref{crle1}, for any $v\in W^{1,q}(\Omega)\,\,(1<q<+\infty)$ with $v\times n=0$ on $\partial\Omega$, it holds that
$$\|\nabla v\|_{L^q(\Omega)}\leq C(\|v\|_{L^q(\Omega)}+\|\div v\|_{L^q(\Omega)}+\|\curl v\|_{L^q(\Omega)}).$$
\end{lemma}

Using Lemmas \ref{crle3}-\ref{crle2}, we further have the following result.
\begin{lemma}   \la{crle5}
Let $D$ be a simply connected domain in $\r^3$ with smooth boundary, and $\Omega$ is the exterior of $D$. For any  $p\in[2,6] $ and    integer $k\geq 0,$  there exists some positive constant $C$ depending only on $p$, $k$ and $D$ such that   every $v\in \{D^{k+1,p}(\Omega)\cap D^{1,2}(\Omega)| v(x,t)\rightarrow 0  \mbox{ as } |x|\rightarrow\infty  \}$  with $v\cdot n|_{\partial\Omega}=0$ or $v\times n|_{\partial\Omega}=0$   satisfies
\be\la{uwkq}\ba\|\nabla v\|_{W^{k,p}(\Omega)}\leq C(\|\div v\|_{W^{k,p}(\Omega)}+\|\curl v\|_{W^{k,p}(\Omega)}+\|\nabla v\|_{L^2(\Omega)}).\ea\ee
\end{lemma}
\begin{proof}First, letting $B_R\triangleq\{x\in \r^3||x|<R\}$ be a ball whose center is at the origin such that $\bar{D}\subset B_R$, it follows from \eqref{g1} that there exists some $C$ depending only on $D$ such that every   $v\in \{  D^{1,2}(\Omega)| v(x,t)\rightarrow 0  \mbox{ as } |x|\rightarrow\infty  \}$ satisfies   for any $p\in [2,6]$ \be\la{ljq04}\ba
 \|v\|_{L^p(B_{2R}\cap\Omega)}
 \leq  C \|v\|_{L^6(\Omega)}  \leq   C\|\na v\|_{L^2(\Omega)},\ea\ee  which together with the standard Sobolev's inequality gives \be \la{ljq05}\ba \|v\|_{ L^4 (\partial \Omega)}&\le C\|v\|_{L^2(B_{2R}\cap\Omega)}+C\|\na v\|_{L^2(B_{2R}\cap\Omega)}\\ &\le C \|\na v\|_{L^2( \Omega)}.\ea\ee

 Then, it suffices to prove \eqref{uwkq} in the case $v\cdot n=0$ since the case $v\times n=0$ can be handled similarly, where we utilize Lemma \ref{crle4} instead of Lemma \ref{crle3}.
We introduce a cut-off function $\eta(x)\in C_c^\infty(B_{2R})$ satisfying $\eta(x)=1$ for $|x|\leq R,$ $\eta(x)=0$ for $|x|\geq 2R,$ $0<\eta(x)<1$ for $R<|x|< 2R,$ and $|\partial^\alpha\eta(x)|<C(R,\alpha)$ for any $0\leq|\alpha|\leq k+1$. Notice that $B_{2R}\cap\Omega$ is a simply connected domain and $\eta v\cdot \overrightarrow{n}=0$ on $\partial B_{2R}\cup \partial \Omega$, consequently,
\be\la{crl 8}\ba  \|\nabla (\eta v)\|_{W^{k,p}( \Omega)} &=\|\nabla (\eta v)\|_{W^{k,p}(B_{2R}\cap\Omega)} \\&\leq C(\|\div (\eta v)\|_{W^{k,p}(B_{2R}\cap\Omega)}+\|\curl (\eta v)\|_{W^{k,p}(B_{2R}\cap\Omega)})\\
&\leq C(\|\div v\|_{W^{k,p}(\Omega)}+\|\curl v\|_{W^{k,p}(\Omega)}+\|v\|_{W^{k,p}(B_{2R}\cap\Omega)}).\ea\ee
On the other hand, the standard $L^p$ estimate implies that
\bnn\la{crl 7}\ba \|\nabla ((1-\eta) v)\|_{W^{k,p}(\r^3)}&\leq C(\|\div ((1-\eta) v)\|_{W^{k,p}(\r^3)}+\|\curl ((1-\eta) v)\|_{W^{k,p}(\r^3)})\\
&\leq C(\|\div v\|_{W^{k,p}(\Omega)}+\|\curl v\|_{W^{k,p}(\Omega)}+\|v\|_{W^{k,p}(B_{2R}\cap\Omega)}),
\ea\enn which together with \eqref{crl 8} leads to
\be\la{crl 9}\ba \|\nabla v\|_{W^{k,p}(\Omega)}\leq C(\|\div v\|_{W^{k,p}(\Omega)}+\|\curl v\|_{W^{k,p}(\Omega)}+\|v\|_{W^{k,p}(B_{2R}\cap\Omega)}).
\ea\ee

Hence, for $k=0$ and $p\in[2,6]$, by Holder's and Gagliardo-Nirenberg's inequalities,
\be\la{crl 10}\ba \|\nabla v\|_{L^p(\Omega)}&\leq C(\|\div v\|_{L^p(\Omega)}+\|\curl v\|_{L^p(\Omega)}+\|v\|_{L^p(B_{2R}\cap\Omega)})\\
&\leq C(\|\div v\|_{L^p(\Omega)}+\|\curl v\|_{L^p(\Omega)}+\|\nabla v\|_{L^2(\Omega)}),
\ea\ee where in the second inequality we have used   \eqref{ljq04}.

 Finally,  choosing  $k=1$ in \eqref{crl 9} gives
\bnn\la{crl 11}\ba \|\nabla v\|_{W^{1,p}(\Omega)}&\leq C(\|\div v\|_{W^{1,p}(\Omega)}+\|\curl v\|_{W^{1,p}(\Omega)}+\|v\|_{W^{1,p}(B_{2R}\cap\Omega)})\\
&\leq C(\|\div v\|_{W^{1,p}(\Omega)}+\|\curl v\|_{W^{1,p}(\Omega)}+\|\nabla v\|_{L^p(\Omega)}+\|\nabla v\|_{L^2(\Omega)}) \\
&\leq C(\|\div v\|_{W^{1,p}(\Omega)}+\|\curl v\|_{W^{1,p}(\Omega)}+\|\nabla v\|_{L^2(\Omega)}) ,
\ea\enn
where in the last inequality we have utilized \eqref{crl 10}. Combining this, \eqref{crl 9}, Gagliardo-Nirenberg's  inequality, and an inductive derivation leads to \eqref{uwkq} and  finishes the proof of Lemma \ref{crle5}.
\end{proof}

Finally, we state the Beale-Kato-Majda type inequality   which was first proved in \cite{bkm,kato} when $\div u\equiv 0$ (see also \cite{hlx} for the case that $\div u\not\equiv 0$).
\begin{lemma}\la{le9}
For $3<q<\infty,$ assume that $u\cdot n=0$ and $\curl u\times n=0,\,\,\,\nabla u\in W^{1,q},$ there is a constant  $C=C(q)$ such that  the following estimate holds
\bnn\ba
\|\na u\|_{L^\infty}\le C\left(\|{\rm div}u\|_{L^\infty}+\|\curl u\|_{L^\infty} \right)\log(e+\|\na^2u\|_{L^q})+C\|\na u\|_{L^2} +C .
\ea\enn
\end{lemma}
\begin{proof} The proof is similar to that of \cite[Lemma 2.7]{CCL1}, and we omit it.
\end{proof}
\subsection{Estimates for $F$, $\curl u$ and $\nabla u$}
Denote
\be \la{dt0}  \text{curl} u \triangleq \nabla\times u ,\quad F\triangleq(2\mu+\lambda)\div u-(P-P(\rho_\infty)),\ee
where $F$  is called the effective viscous flux and plays an important role in our analysis. Now we give some a priori estimates for $F$, $\curl u$ and $\nabla u$, which will be used frequently later.
\begin{lemma}   \la{le3}
 Assume $\Omega$ is an exterior domain of some simply connected bounded domain in $\r^3$ and its boundary $\partial\Omega$ is smooth. Let $(\rho,u)$ be a smooth solution of \eqref{a1} in $\Omega$ with slip condition. Then for any $p\in[2,6]$ and $ q\in (1,\infty),$ there exists a positive constant $C$ which may depend only on $p, q, \mu,\lambda$ and $\Omega$ such that
%\be\la{tdu1}\ba
%\|\nabla u\|_{L^q}\leq C(\|F\|_{L^q}+\|\curl u\|_{L^q}+\|P-P(\rho_\infty)\|_{L^q}),\,\,\,1<q<3,\\
%\ea\ee
%\be\la{tdu11}\ba
%\|\nabla u\|_{L^q}\leq C(\|F\|_{L^q}+\|\curl u\|_{L^q}+\|P-P(\rho_\infty)\|_{L^q}+\|\nabla u\|_{L^2}),\,\,\,3\leq q<+\infty,
%\ea\ee
\be\la{h19}\ba
\|\nabla F\|_{L^q}\leq C\|\rho\dot{u}\|_{L^q},
\ea\ee
\be\la{zh19}\ba
\|\nabla\curl u\|_{L^p}\leq C(\|\rho\dot{u}\|_{L^p}+\|\rho\dot{u}\|_{L^2}+\|\nabla u\|_{L^2}),
\ea\ee
\be\la{h20}\ba
\|F\|_{L^p}&\leq C\|\rho\dot{u}\|_{L^2}^{(3p-6)/(2p)}(\|\nabla u\|_{L^2}+\|P-P(\rho_\infty)\|_{L^2})^{(6-p)/(2p)},
\ea\ee
\be\la{zh20}\ba
\|\curl u\|_{L^p}&\leq C\|\rho\dot{u}\|_{L^2}^{(3p-6)/(2p)}\|\nabla u\|_{L^2}^{(6-p)/(2p)}+C\|\nabla u\|_{L^2},
\ea\ee
  %\be\la{hh20}\ba \|F\|_{L^p}+\|\curl u\|_{L^p}\leq C(\|\rho\dot{u}\|_{L^2}+\|\nabla u\|_{L^2}+\|P-P(\rho_\infty)\|_{L^2}), \ea\ee
\be\la{h18}\ba
\|\nabla u\|_{L^p}\leq& C\|\nabla u\|_{L^2}^{(6-p)/(2p)}(\|\rho\dot{u}\|_{L^2} +\|P-P(\rho_\infty)\|_{L^6})^{(3p-6)/(2p)}\\&+C\|\nabla u\|_{L^2}.
\ea\ee
\end{lemma}
\begin{proof}
%The inequality \eqref{tdu1} and \eqref{tdu11} are direct results of Lemma \ref{crle1}, since $u\cdot n=0$ on $\partial\Omega$ and $F=(\lambda+2\mu)\div u-(P-P(\rho_\infty))$.
By $(\ref{a1})_2 $, it is easy to find that $F$ satisfies
$$\int\nabla F\cdot\nabla\eta dx=\int\rho\dot{u}\cdot\nabla\eta dx,\,\,\forall\eta\in C_0^{\infty}(\r^3).$$
Consequently, by Lemma \ref{zhle},
\be\la{x266}\ba
\|\nabla F\|_{L^q}\leq C\|\rho\dot{u}\|_{L^q},
\ea\ee
and for any integer $k\geq 1$,
\be\la{x2666}\ba
\|\nabla F\|_{W^{k,q}}\leq C\|\rho\dot{u}\|_{W^{k,q}}.
\ea\ee

On the other hand, $(\ref{a1})_2 $ can be rewritten as $\mu\nabla\times\curl u=\nabla F-\rho\dot{u}.$
Notice that $\curl u\times n=0$ on $\partial\Omega$ and $\div(\nabla\times\curl u)=0,$ we deduce from Lemmas \ref{crle2}-\ref{crle5} and \eqref{x266} that
\be\la{x267}\ba
\|\nabla\curl u\|_{L^q}&\leq C(\|\nabla\times\curl u\|_{L^q}+\|\curl u\|_{L^q})\\
&\leq C(\|\rho\dot{u}\|_{L^q}+\|\curl u\|_{L^q}).
\ea\ee
and for any integer $k\geq 1$,
\be\la{x268}\ba
\|\nabla\curl u\|_{W^{k,p}}&\leq C(\|\nabla\times\curl u\|_{W^{k,p}}+\|\nabla\curl u\|_{L^2})\\
&\leq C(\|\rho\dot{u}\|_{W^{k,p}}+\|\rho\dot{u}\|_{L^2}+\|\nabla u\|_{L^2}),
\ea\ee
Therefore, by Gagliardo-Nirenberg's inequality and \eqref{x267}, for $p\in[2,6]$,
\bnn\ba
\|\nabla\curl u\|_{L^p}&\le C(\|\rho\dot{u}\|_{L^p}+\|\curl u\|_{L^p} ) \\
&\le C(\|\rho\dot{u}\|_{L^p}+\|\nabla\curl u\|_{L^2}+\|\curl u\|_{L^2}) \\
&\le C(\|\rho\dot{u}\|_{L^p}+\|\rho\dot{u}\|_{L^2}+\|\curl u\|_{L^2}) \\
&\le C(\|\rho\dot{u}\|_{L^p}+\|\rho\dot{u}\|_{L^2}+\|\nabla u\|_{L^2}),
\ea\enn
which gives \eqref{zh19}.

Furthermore, it follows from  Gagliardo-Nirenberg's inequality, \eqref{g1} and \eqref{h19} that for $p\in[2,6]$,
\be\la{x2610}\ba
\|F\|_{L^p}&\leq C\|F\|_{L^2}^{(6-p)/2p}\|\nabla F\|_{L^2}^{(3p-6)/2p}\\
&\leq C\|\rho\dot{u}\|_{L^2}^{(3p-6)/2p}(\|\nabla u\|_{L^2}+\|P-P(\rho_\infty)\|_{L^2})^{(6-p)/2p}.
\ea\ee
%which immediately  leads to \be\la{x2611}\ba \|F\|_{L^p}\leq C(\|\rho\dot{u}\|_{L^2}+\|\nabla u\|_{L^2}+\|P-P(\rho_\infty)\|_{L^2}). \ea\ee
Similarly,
\be\la{x2612}\ba
\|\curl u\|_{L^p}&\leq C\|\curl u\|_{L^2}^{(6-p)/(2p)}\|\nabla\curl u\|_{L^2}^{(3p-6)/(2p)}\\
&\leq C(\|\rho\dot{u}\|_{L^2}+\|\nabla u\|_{L^2})^{(3p-6)/(2p)}\|\nabla u\|_{L^2}^{(6-p)/(2p)}\\
&\leq C\|\rho\dot{u}\|_{L^2}^{(3p-6)/(2p)}\|\nabla u\|_{L^2}^{(6-p)/(2p)}+C\|\nabla u\|_{L^2}.
\ea\ee
A combination of \eqref{x2610} and \eqref{x2612} yields \eqref{h20}.

Finally, by virtue of Lemma \ref{crle1}, \eqref{g1}, \eqref{h19} and \eqref{zh20}, it indicates that
\bnn\ba
 \|\nabla u\|_{L^p}&\le C\|\nabla u\|_{L^2}^{(6-p)/(2p)}\|\nabla u\|_{L^6}^{(3p-6)/(2p)} \\
&\le C\|\nabla u\|_{L^2}^{(6-p)/(2p)}(\|\rho\dot{u}\|_{L^2}+\|P-P(\rho_\infty)\|_{L^6})^{(3p-6)/(2p)}+C\|\nabla u\|_{L^2},
\ea\enn where in the second inequality we have used
\bnn\ba
 \|\nabla u\|_{L^6}  &\le C (\|\div u\|_{L^6}+\|\curl u\|_{L^6}+\|\nabla u\|_{L^2}) \\
&\le C (\|F\|_{L^6}+\|\curl u\|_{L^6}+\|P-P(\rho_\infty)\|_{L^6}+\|\nabla u\|_{L^2})  \\
&\le C (\|\rho\dot{u}\|_{L^2}+\|P-P(\rho_\infty)\|_{L^6} + \|\nabla u\|_{L^2}),
\ea\enn due to both \eqref{x2610} and \eqref{x2612}  with $p=6.$
The proof of Lemma \ref{le3} is finished.
\end{proof}
\section{\la{se3} A priori estimates(I): lower order estimates}

In this section,   Assume $\Omega$ is always the exterior of a simply connected bounded domain in $\r^3$. Choosing a positive real number $R$  such that $\bar{D}\subset B_R$, one can extend the unit outer normal $n$ to $\Omega$ such that \be\la{ljq10} n\in C^3(\bar{\Omega}),\quad n\equiv 0  \mbox{ on }\r^3\setminus  B_{2R}.\ee

 Let $T>0$ be a fixed time and $(\n,u)$ be
the smooth solution to (\ref{a1})-(\ref{ch2})  on
$\Omega \times (0,T]$  with smooth initial
data $(\n_0,u_0)$ satisfying satisfying (\ref{dt1}) and (\ref{dt2}). We are going to establish some necessary a priori bounds for smooth solutions of the problem (\ref{a1})-(\ref{ch2}) which can extend the local classical solution guaranteed by Lemma \ref{loc1} to be a global one.

  In what follows, we will use the convention that $C$ may denote a generic positive constant depending on $\mu ,  \lambda ,   \ga ,  a ,  \on, \rho_\infty,   \Omega$ and $M$, and use $C(\delta)$ to emphasize that $C$ depends on $\delta.$  We have the following standard energy estimate for $(\rho,u)$ and preliminary $L^{2}$ bounds for $\nabla u$ and $\rho^{\frac{1}{2}}\dot{u}.$

%We have the following key a priori estimates on $(\n,u)$.
%\begin{proposition}\la{pr1}  Under  the conditions of Theorem \ref{th1},
%     there exists some  positive constant  $\ve$
%    depending    on  $\mu ,  \lambda ,   \ga ,  a ,  \on, s,$ $N_0,$ and $M$  such that if
%       $(\n,u)$  is a smooth solution of
%       (\ref{a1})-(\ref{h2})  on $\O \times (0,T] $
%        satisfying
% \be\la{z1}
% \sup\limits_{
% \O \times [0,T]}\n\le 2\bar{\n},\quad
%     A_1(T) + A_2(T) \le 2C_0^{1/2},
%   \ee
%     the following estimates hold
%        \be\la{z2}
% \sup\limits_{\O \times [0,T]}\n\le 7\bar{\n}/4, \quad
%     A_1(T) + A_2(T) +\int_0^T\si \|P\|_{L^2}^2dt\le  C_0^{1/2},
%  \ee
%   provided $C_0\le \ve.$
%\end{proposition}

%The proof of Proposition \ref{pr1} will be postponed to the end of this section.

%In the following, we will use the convention that $C$ denotes a
%generic positive constant
% depending  on $\mu$, $\lambda$, $\gamma$, $a$,
%$\bar{\n},$  $s,$ $N_0,$ and $M$, and  use $C(\al)$ to emphasize
%that $C$ depends on $\al.$

%We begin with the following   standard   energy estimate for
%$(\n,u)$ and preliminary    $L^2$ bounds for $\nabla u$ and
%$\n\dot{u}$.
\begin{lemma}\la{le2}
 Let $(\n,u)$ be a smooth solution of
 \eqref{a1}--\eqref{ch2} on $\O \times (0,T]. $
  Then there is a positive constant
  $C $ depending only  on $\mu,$  $\lambda,$ $\gamma$ and $\Omega$  such that
\be \la{a16} \sup_{0\le t\le T}
\left( \|\n^\frac{1}{2} u \|_{L^{2}}^{2}+\|G(\rho)\|_{L^{1}}\right) + \int_{0}^{T}\|\nabla u\|_{L^{2}}^{2}dt dt\le CC_0. \ee
\end{lemma}

\begin{proof}
%First,
%the standard energy inequality reads:
%\bnn
% \sup\limits_{0\le t\le T}\int\left(
%\frac{1}{2}\n|u|^2+\frac{P}{\ga-1}\right)dx+\int_0^T \int\left( \mu |\na u|^2+(\mu+\lambda )(\div u)^2\right)dxdt \le  C_0,
%\enn which together with \eqref{h3} shows \eqref{a16}.
Multiplying $(\ref{a1})_1 $ by $G'(\rho)$ and taking advantage of slip boundary condition \eqref{ch1} show that

\be\la{m0} \ba \left(\int G(\rho)dx\right)_t + \int(P-P(\rho_\infty))\div udx=0.\ea \ee

Note that $-\Delta u=-\nabla\div u+\nabla\times\curl u,$ we
rewrite $(\ref{a1})_2 $ as
\be\la{m1} \ba
\rho \dot{u} - (\lambda + 2\mu)\nabla\div u+\mu\nabla\times\curl u + \nabla(P - P(\rho_\infty))=0.
\ea \ee
Multiplying \eqref{m1} by $u$, along with \eqref{m0} and \eqref{ch1}, we have
\be\la{m8} \ba
\left(\frac12\int\rho |u|^{2}dx  + \int G(\rho)dx\right)_t + (\lambda + 2\mu)\int(\div u)^{2}dx + \mu\int|\curl u|^{2}dx=0,
\ea \ee
which together with \eqref{ljq01} gives \eqref{a16} and finishes the proof of Lemma \ref{le2}.
\end{proof}

For $\si(t)\triangleq\min\{1,t \},$  we define
 \be\la{As1}
  A_1(T) \triangleq \sup_{   0\le t\le T  }\left(\sigma\|\nabla u\|_{L^2}^2\right) + \int_0^{T} \sigma\int
 \n|\dot{u} |^2 dxdt,
  \ee
\be \la{As2}
  A_2(T)  \triangleq\sup_{  0\le t\le T   }\sigma^3\int\n|\dot{u}|^2dx + \int_0^{T}\int
  \sigma^3|\nabla\dot{u}|^2dxdt,
\ee
and
\be \la{As3}
  A_3(T)  \triangleq\sup_{  0\le t\le T   }\int\rho|u|^3dx .
\ee

The rest of section is devoted to proving the following proposition, which guarantees the existence of a global classical solution of \eqref{a1}--\eqref{ch2}.
\begin{proposition}\la{pr1}   Under  the conditions of Theorem \ref{th1},
   there exists a  positive constant  $\ve$
    depending    on  $\mu$, $\lambda$, $a$, $\ga$, $\on,$ $\rho_\infty$, $\Omega$ and $M$  such that if
       $(\rho,u)$  is a smooth solution of
       \eqref{a1}--\eqref{ch2}  on $\Omega\times (0,T] $
        satisfying
 \be\la{zz1}
 \sup\limits_{
 \Omega\times [0,T]}\rho\le 2\bar{\rho},\quad
     A_1(T) + A_2(T) \le 2C_0^{1/2},\quad A_3(\sigma(T))\leq 2C_0^{\frac{1}{4}},
  \ee
 then
        \be\la{zz2}
 \sup\limits_{\Omega\times [0,T]}\rho\le 7\bar{\rho}/4, \quad
     A_1(T) + A_2(T) \le  C_0^{1/2},\quad A_3(\sigma(T))\leq C_0^{\frac{1}{4}},
  \ee
   provided $C_0\le \ve.$
\end{proposition}
\begin{proof}Proposition \ref{pr1} is a direct consequence of Lemmas \ref{nzc1}, \ref{le5} and \ref{le7} below.
\end{proof}
The following Lemmas \ref{xcrle1}--\ref{le7} will be proven under the assumption \eqref{zz1}.
\begin{lemma}\la{xcrle1}
 Suppose $(\n,u)$ is a smooth solution of
 \eqref{a1}-\eqref{ch2} on $\O \times (0,T]$ satisfying \eqref{zz1}. Then there exists a positive constant
  $C $ depending only  on $\mu$, $\lambda$, $a$, $\ga$, $\on,$ $\rho_\infty$ and $\Omega$  such that
  \be\la{h14}
  A_1(T) \le  C C_0 + C\int_0^{T}\int\sigma|\nabla u|^3dx dt,
  \ee
 and
  \be\la{h15}
    A_2(T)
    \le   C C_0 + CA_1(T)  + C\int_0^{T}\int \sigma^3 |\nabla u|^4dx dt.
   \ee
\end{lemma}
\begin{proof}
The proof is motivated by  \cite{hlx1,CCL1,H3}.
 First, the following two identities will be used more than once:
 \be\la{divug}
   \div (u\cdot \nabla u)=u\cdot \nabla (\div u)+\nabla u :\nabla u,
   \ee
 \be\la{curlug}
   \curl (u\cdot \nabla u)=u\cdot \nabla (\curl u)+\nabla u^i \times\nabla_i u.
   \ee

 Next, for $m\ge 0,$ multiplying $(\ref{a1})_2 $ by
$\sigma^m \dot{u}$ yields
\be\la{I0} \ba  \int \sigma^m \rho|\dot{u}|^2dx &= -\int\sigma^m \dot{u}\cdot\nabla Pdx + (\lambda+2\mu)\int\sigma^m \nabla\div u\cdot\dot{u}dx \\
&\quad - \mu\int\sigma^m \nabla\times\curl u\cdot\dot{u}dx \\
& \triangleq \sum_{i=1}^{3}I_i. \ea \ee
We will estimate  the three terms in the last equality. It follows  from the slip boundary condition $u\cdot n|_{\partial\Omega }=0$   that \be\la{bdd1}u\cdot\nabla u\cdot n=-u\cdot\nabla n\cdot u  \, \mbox{ on }\partial\Omega.\ee
%where $\kappa(x,u)$, geometrically, is the normal curvature along $u$ direction at the point $x\in\partial\Omega$.
For $I_1$, we check that
\be\la{I10} \ba
I_1 = & - \int \sigma^m \dot{u}\cdot\nabla Pdx \\
= & -\int\sigma^m u_{t}\cdot\nabla(P-P(\rho_\infty))dx
- \int\sigma^m u\cdot\nabla u\cdot\nabla Pdx \\
= & \left(\int\sigma^m(P-P(\rho_\infty))\div u\, dx\right)_{t} - m\sigma^{m-1}\sigma'\int(P-P(\rho_\infty))\div u\,dx  \\
&+ \int\sigma^{m}P\nabla u:\nabla u dx+ (\gamma-1)\int\sigma^{m}P(\div u)^{2}dx + \int_{\partial\Omega}\sigma^{m}Pu\cdot\nabla n\cdot u ds, \ea \ee
where we have used \eqref{bdd1}, \eqref{divug} and the relation
 $$P_t+\div(Pu)+(\gamma-1)P\div u=0.$$

For the boundary term in the last  of \eqref{I10}, one has by \eqref{ljq05}
\bnn \ba
 \int_{\partial\Omega}\sigma^{m}Pu\cdot\nabla n\cdot uds
&\le C(\bar{\rho})\sigma^{m} \|u\|_{L^2(\partial\Omega)}^2   \\
&\le C(\bar{\rho})\sigma^{m}\|\nabla u\|_{L^{2}}^{2}.
\ea  \enn

As a result, together with \eqref{a16},
\be\la{m20} \ba
I_1\leq &\left(\int\sigma^m(P-P(\rho_\infty))\div u\, dx\right)_{t} + C(\bar{\rho})\|\nabla u\|_{L^{2}}^{2} + C(\bar{\rho})m\sigma^{m-1}\sigma'C_{0}.
\ea \ee
Similarly, by \eqref{divug},
\be\la{I2} \ba
I_2 & =  (\lambda+2\mu)\int\sigma^m \nabla\div u\cdot\dot{u}dx \\
& = (\lambda+2\mu)\int_{\partial\Omega}\sigma^m\div u\,(\dot{u}\cdot n)ds - (\lambda+2\mu)\int\sigma^m\div u\,\div \dot{u}dx  \\
& = (\lambda+2\mu)\int_{\partial\Omega}\sigma^m\div u\,(u\cdot\nabla u\cdot n)ds - \frac{\lambda+2\mu}{2}\left(\int\sigma^{m}(\div u)^{2}dx\right)_{t} \\
&\quad - (\lambda+2\mu)\int\sigma^m\div u\,\div(u\cdot\nabla u)dx + \frac{m(\lambda+2\mu)}{2}\sigma^{m-1}\sigma'\int(\div u)^{2}dx \\
& = (\lambda+2\mu)\int_{\partial\Omega}\sigma^m\div u\,(u\cdot\nabla u\cdot n)ds - \frac{\lambda+2\mu}{2}\left(\int\sigma^{m}(\div u)^{2}dx\right)_{t} \\
&\quad +\frac{\lambda+2\mu}{2}\int\sigma^{m}(\div u)^{3}dx- (\lambda+2\mu)\int\sigma^m\div u\,\nabla u:\nabla udx  \\
&\quad + \frac{m(\lambda+2\mu)}{2}\sigma^{m-1}\sigma'\int(\div u)^{2}dx.
\ea\ee

On the other hand, it follows from H\"{o}lder's inequality, \eqref{ljq05}, and \eqref{h19} that
\bnn \ba
&\left|\int_{\partial\Omega}\div u\,(u\cdot\nabla u\cdot n)ds\right| \\
&=\left|\int_{\partial\Omega}\div u\,(u\cdot\nabla n\cdot u)ds\right|\\
& \leq\frac{1}{\lambda +2\mu}\left|\int_{\partial\Omega} Fu\cdot\nabla n\cdot uds\right|+\frac{1}{\lambda +2\mu}\left|\int_{\partial\Omega}(P-P(\rho_\infty))u\cdot\nabla n\cdot uds\right|\\
& \leq C\left(\|F\|_{L^{4}( \partial\Omega)}\|u\|_{L^{4}( \partial\Omega)}^{2}+ \|u\|_{L^{2}( \partial\Omega)}^{2}\right) \\
& \leq C\left(\|\nabla F\|_{L^{2}}\|\nabla u\|_{L^{2}}^2+\|\nabla u\|_{L^{2}}^2\right)\\
& \leq\frac{1}{2}\|\rho^{\frac{1}{2}}\dot{u}\|_{L^{2}}^{2}+C(\|\nabla u\|_{L^{2}}^{2}+\|\nabla u\|_{L^{2}}^{4}).
\ea  \enn
Hence,
\be\la{I21} \ba
I_2 & \leq - \frac{\lambda+2\mu}{2}\left(\int\sigma^{m}(\div u)^{2}dx\right)_t+Cm\sigma^{m-1}\sigma'\|\nabla u\|_{L^{2}}^{2}+C\int\sigma^{m}|\nabla u|^{3}dx\\
&\quad +\frac{1}{2}\sigma^{m}\|\rho^{\frac{1}{2}}\dot{u}\|_{L^{2}}^{2}+C\sigma^{m}\|\nabla u\|_{L^{2}}^{4}+C\|\nabla u\|_{L^{2}}^{2}.\ea\ee
For $I_3$, a straightforward computation along with \eqref{curlug} gives that
\be\la{I3}\ba
I_3 & = -\mu\int\sigma^{m}\nabla\times\curl u\cdot\dot{u}dx \\
& = -\mu\int\sigma^{m}\div(\curl u\times\dot{u})dx - \mu\int\sigma^{m}\curl u\cdot\curl\dot{u}dx \\
& = -\frac{\mu}{2}\left(\int\sigma^{m}|\curl u|^{2}dx\right)_t + \frac{\mu m}{2}\sigma^{m-1}\sigma'\int|\curl u|^{2}dx \\
& \quad - \mu\int\sigma^{m}\curl u\cdot\curl(u\cdot\nabla u)dx \\
& = -\frac{\mu}{2}\left(\int\sigma^{m}|\curl u|^{2}dx\right)_t + \frac{\mu m}{2}\sigma^{m-1}\sigma'\int|\curl u|^{2}dx   \\
& \quad - \mu\int\sigma^{m}(\nabla u^{i}\times\nabla_i u)\cdot\curl udx-\frac{\mu}{2}\int_{\partial\Omega}\sigma^{m}|\curl u|^{2}u\cdot nds  \\
& \quad + \frac{\mu}{2}\int\sigma^{m}|\curl u|^{2}\div udx\\
& \leq -\frac{\mu}{2}\left(\int\sigma^{m}|\curl u|^{2}dx\right)_t + C\sigma^{m-1}\sigma'\|\nabla u\|_{L^{2}}^{2} + C\sigma^{m}\|\nabla u\|_{L^{3}}^{3}.
\ea \ee
Together with the estimates of $I_1$, $I_2$ and $I_3$ above, it follows from $(\ref{I0})$ that
\be\la{I4}\ba
&\left((\lambda+2\mu)\int\sigma^{m}(\div u)^{2}dx+\mu\int\sigma^{m}|\curl u|^{2}dx\right)_{t}\\
&\quad-2\left(\int\sigma^{m}(P-P(\rho_\infty))\div udx\right)_{t}+\int\sigma^{m}\rho|\dot{u}|^{2}dx \\
& \leq C(\bar{\rho})m\sigma^{m-1}\sigma'C_{0}+C\sigma^{m}\|\nabla u\|_{L^{2}}^{4}+C(\bar{\rho})\|\nabla u\|_{L^{2}}^{2}+C\sigma^{m}\|\nabla u\|_{L^{3}}^{3},
\ea \ee which together with   Young's inequality, Lemma \ref{crle1}, \eqref{a16}, \eqref{a16} and \eqref{zz1} gives \eqref{h14}.

Now it remains to prove \eqref{h15}. Rewrite $ (\ref{a1})_2 $ as
\be\la{xdy1}\ba
\rho\dot{u}=\nabla F - \mu\nabla\times\curl u.
\ea \ee
Imposing $ \sigma^{m}\dot{u}^{j}[\pa/\pa t+\div
(u\cdot)] $ on $ (\ref{xdy1})^j ,$  summing with respect to $j$, and integrating over $\Omega,$
we obtain
\be\la{ax1}\ba &\left(\frac{\sigma^{m}}{2}\int\rho|\dot{u}|^{2}dx\right)_t-\frac{m}{2}\sigma^{m-1}\sigma'\int\rho|\dot{u}|^{2}dx \\
& = -\int\sigma^{m}\dot{u}^{j}\div(\rho\dot{u}^{j}u)dx - \frac{1}{2}\int\sigma^{m}\rho_t|\dot{u}|^{2}dx\\
&\quad+\int\sigma^{m}(\dot{u}\cdot\nabla F_t+\dot{u}^{j}\div(u\partial_jF))dx  \\
&\quad-\mu\int\sigma^{m}(\dot{u}\cdot\nabla\times\curl u_t+\dot{u}^{j}\div(u(\nabla\times\curl u)^{j})dx \\
& = \int\sigma^{m}(\dot{u}\cdot\nabla F_t+\dot{u}^{j}\div(u\partial_jF))dx\\
&\quad+\mu\int\sigma^{m}(-\dot{u}\cdot\nabla\times\curl u_t-\dot{u}^{j}\div(u(\nabla\times\curl u)^{j})dx \\
& =J_1+J_2,
\ea\ee
where we have utilized $(\ref{a1})_1 $ and the boundary condition $u\cdot n=0$.

Let's estimate $J_1$ and $J_2$. For $J_1$, by Gagliardo-Nirenberg's, Young's and H\"{o}lder's inequalities, we deduce from \eqref{divug}, $(\ref{a1})_1$ and \eqref{h19} that
\be\la{ax2}\ba J_1& = \int\sigma^{m}\dot{u}\cdot\nabla F_tdx+\int\sigma^{m}\dot{u}^{j}\div(u\partial_jF)dx \\
& = \int_{\partial\Omega}\sigma^{m}F_t\dot{u}\cdot nds - \int\sigma^{m}F_t\div\dot{u}dx + \int\sigma^{m}\dot{u}\cdot\nabla\div(uF)dx \\
& \quad - \sigma^{m}\int\dot{u}^{j}\div(\nabla_juF)dx \\
& = \int_{\partial\Omega}\sigma^{m}F_t\dot{u}\cdot nds +\int_{\partial\Omega}\sigma^{m}(\nabla F\cdot u)\dot{u}\cdot nds+\frac{1}{\lambda+2\mu}\int_{\partial\Omega}\sigma^{m}F^2\dot{u}\cdot nds \\
& \quad - (\lambda+2\mu)\int\sigma^{m}(\div\dot{u})^{2}dx+ (\lambda+2\mu)\int\sigma^{m}\div\dot{u}\nabla u:\nabla u dx\\
&\quad -\gamma\int\sigma^{m} P\div\dot{u}\,\div udx-\int\sigma^{m}F\div \dot{u}\,\div udx-\int\sigma^{m}\dot{u}\cdot\nabla u\cdot\nabla Fdx \\
&\quad - \frac{1}{\lambda+2\mu}\int\sigma^mF\dot{u}\cdot\nabla Fdx+\frac{1}{\lambda+2\mu}\int\sigma^m(P-P(\rho_\infty))\nabla F\cdot\dot{u}dx\\
&\quad +\frac{1}{\lambda+2\mu}\int\sigma^m(P-P(\rho_\infty))F\div\dot{u}dx
\\& \triangleq \sum_{i=1}^3\ti J_1^i - (\lambda+2\mu)\int\sigma^{m}(\div\dot{u})^{2}dx+\sigma^{m}\sum_{j=1}^7J_1^j.\ea\ee

\be\ba \sum_{j=1}^7|J_1^j|
& \leq C\|\nabla \dot{u}\|_{L^2} (\|\nabla u\|_{L^4}^2 +\|\nabla u\|_{L^2}) +C\|\dot{u}\|_{L^6} \|\nabla u\|_{L^2}\|\rho\dot{u}\|_{L^2}^\frac{1}{2} \|\rho\dot{u}\|_{L^6}^\frac{1}{2} \\&\quad +C (\|F\|_{L^2}\|\rho\dot{u}\|_{L^2}^\frac{1}{2}\|\rho\dot{u}\|_{L^6}^\frac{1}{2}\|\dot{u}\|_{L^6}
+\|\nabla F\|_{L^2}\|P-P(\rho_\infty)\|_{L^3}\|\dot{u}\|_{L^6})\\
&\quad +C \|F\|_{L^6}\|P-P(\rho_\infty)\|_{L^3}\|\nabla\dot{u}\|_{L^2}\\
& \leq \frac{\delta}{8} \|\nabla\dot{u}\|_{L^2}^2+C(\delta) \|\rho^{\frac{1}{2}}\dot{u}\|_{L^2}^2\|\nabla u\|_{L^2}^4+C(\delta) \|\rho^{\frac{1}{2}}\dot{u}\|_{L^2}^2\\
&\quad+C(\delta) \|\nabla u\|_{L^2}^2+C(\delta)\sigma^m\|\nabla u\|_{L^4}^4,
\ea\ee
where $\delta>0$ will be determined later.

Next, it follows  from the slip boundary condition $u\cdot n|_{\partial\Omega }=0$   that for $u^{\perp}\triangleq-u\times n,$
\be\la{bdd2}u=u^{\perp}\times n \mbox{ on }\partial\Omega. \ee
%To estimate the three boundary terms in the last inequality, we will apply  Gagliardo-Nirenberg's, Young's, H\"{o}lder's inequalities, Lemma \ref{le3},  and the equalities \eqref{bdd2}, \eqref{bdd1}.
For the first boundary term $\ti J^1_1,$ denoting $h\triangleq u\cdot(\nabla n+\nabla n^{tr}),$   we have by  \eqref{bdd1}, Lemma \ref{le3}, and Gagliardo-Nirenberg's, Young's, and H\"{o}lder's inequalities,
\be\la{ax3}\ba
&\ti J^1_1+\left(\int_{\partial\Omega}\sigma^{m}(u\cdot\nabla n\cdot u)Fds\right)_t\\&=-\int_{\partial\Omega}\sigma^{m}F_t\,(u\cdot\nabla n\cdot u)ds + \left(\int_{\partial\Omega}\sigma^{m}(u\cdot\nabla n\cdot u)Fds\right)_t\\
& =  \int_{\partial\Omega}\sigma^{m}Fh\cdot u_tds +m\sigma^{m-1}\sigma'\int_{\partial\Omega}(u\cdot\nabla n\cdot u)Fds \\
& =   \int_{\partial\Omega}\sigma^{m}Fh\cdot\dot{u}ds -\int_{\partial\Omega}\sigma^{m}Fh\cdot((u^{\perp}\times n)\cdot\nabla u)ds\\
&\quad+m\sigma^{m-1}\sigma'\int_{\partial\Omega}(u\cdot\nabla n\cdot u)Fds   \\
& =  \int_{\partial\Omega}\sigma^{m}Fh\cdot\dot{u}ds - \int\sigma^{m}\nabla u^i\times u^\perp\cdot\nabla(Fh^i)dx\\
&\quad+  \int\sigma^{m}Fh^i\nabla\times u^\perp \cdot\nabla u^{i}dx+m\sigma^{m-1}\sigma'\int_{\partial\Omega}(u\cdot\nabla n\cdot u)Fds,\ea\ee
where in the last inequality we have used the following key observation due to \cite{CCL1}
\be\ba
&  -\int_{\partial\Omega} Fh\cdot((u^{\perp}\times n)\cdot\nabla u)ds   \\& =   -  \int \div(Fh^{i}\nabla u^{i}\times u^{\perp})dx \\
& =   - \int \nabla u^i\times u^\perp\cdot\nabla(Fh^i)dx + \int Fh^i\nabla\times u^\perp \cdot\nabla u^{i}dx. \ea\ee

It follows from \eqref{ljq05} and \eqref{h19} that
\be\ba \int_{\partial\Omega} | Fh\cdot\dot{u}|ds&\le C\|F\|_{L^4(\partial\Omega)}\|u\|_{L^4(\partial\Omega)}\|\dot u\|_{L^2(\partial\Omega)}\\& \le C\|\na F\|_{L^2( \Omega)}\|\na u\|_{L^2( \Omega)} \|\na \dot u\|_{L^2( \Omega)} \\& \le \frac{\de}{8}\|\na \dot u\|_{L^2( \Omega)}^2+C(\de)\|\n^{1/2} \dot u \|_{L^2( \Omega)}^2\|\na u\|_{L^2( \Omega)}^2  \ea\ee

\be\ba &\left| \int \nabla u^i\times u^\perp\cdot\nabla(Fh^i)dx\right| + \left|\int Fh^i\nabla\times u^\perp \cdot\nabla u^{i}dx\right|\\&
\le C\|\na F\|_{L^6}\|\na u\|_{L^2}\|u\|_{L^6}^2+C\|  F\|_{L^6}\|\na u\|_{L^2}\|u\|_{L^6}^2\\&\quad+C\|  F\|_{L^6}\|\na u\|^2_{L^3}\|u\|_{L^6}\\& \le \frac{\de}{8}\|\na \dot u\|_{L^2 }^2+C(\de)\|\n^{1/2} \dot u \|_{L^2 }^2\|\na u\|_{L^2 }^2  +C(\de) \|\na u\|_{L^2 }^6  +C(\de) \|\na u\|_{L^3 }^4  . \ea\ee

\be\la{ljq11}\ba  \left|\int_{\partial\Omega}(u\cdot\nabla n\cdot u)Fds\right| & \le C\|F\|_{L^4(\partial\Omega)}\|u\|^2_{L^4(\partial\Omega)}\\& \le C\|\na F\|_{L^2( \Omega)}\|\na u\|^2_{L^2( \Omega)}\\& \le \frac{\de}{8}\|\na \dot u\|_{L^2 }^2   +C(\de) \|\na u\|_{L^2 }^4 \ea\ee

\be\ba \ti J^1_1\le&-\left(\int_{\partial\Omega}\sigma^{m}(u\cdot\nabla n\cdot u)Fds\right)_t
 +\frac{\delta}{8}\sigma^{m}\|\nabla\dot{u}\|_{L^2}^2 \\&+Cm\sigma^{m-1}\sigma' \|\rho^{\frac{1}{2}}\dot{u}\|_{L^2}^{2} +C(\delta)\|\rho^{\frac{1}{2}}\dot{u}\|_{L^2}^{2}\|\nabla u\|_{L^2}^{2}\\
& +C(\delta)\sigma^{m}( \|\nabla u\|_{L^4}^{4}+\|\nabla u\|_{L^2}^{2}+\|\nabla u\|_{L^2}^{6}).
\ea\ee
For the second boundary term $\ti J^2_1$,
\be\la{cax31}\ba  \ti J^2_1
&=-\int_{\partial\Omega} (\nabla F\cdot u) (u\cdot\nabla n\cdot u)ds \\
&=-\int_{\partial\Omega} (u\cdot\nabla n\cdot u)\nabla F\cdot(u^{\perp}\times n)ds \\
&=-\int_{\partial\Omega} (u\cdot\nabla n\cdot u)(\nabla F\times u^{\perp})\cdot nds \\
&=-\int \div[(u\cdot\nabla n\cdot u)(\nabla F\times u^{\perp})]dx\\
&=-\int \nabla(u\cdot\nabla n\cdot u)\cdot(\nabla F\times u^{\perp})dx+\int (u\cdot\nabla n\cdot u)(\nabla\times u^{\perp}\cdot \nabla F)dx\\
&\le C (\|\nabla u\|_{L^2}\|u\|_{L^6}^2\|\nabla F\|_{L^6}+\|u\|_{L^6}^3\|\nabla F\|_{L^2}+\|\nabla F\|_{L^2}\|\nabla u\|_{L^2}^3%+\|\nabla \dot{u}\|_{L^2}\|\nabla u\|_{L^2}^3
)\\
&\leq \frac{\delta}{8} \|\nabla\dot{u}\|_{L^2}^2+C(\delta) (\|\nabla u\|_{L^2}^2\|\rho^{\frac{1}{2}}\dot{u}\|_{L^2}^2+\|\nabla u\|_{L^2}^2+\|\nabla u\|_{L^2}^6).
\ea\ee
Similarly, for the last boundary term $\ti J^3_1$,
\be\la{b1}\ba\ti J^3_1&=-\frac{1}{\lambda+2\mu}\int_{\partial\Omega}\sigma^{m}(u\cdot\nabla n\cdot u)F^2ds\\
&\leq C\sigma^{m} \|u\|_{L^4(\partial\Omega)}^2\|F\|_{L^4(\partial\Omega)}^2 \\
&\leq C\sigma^{m}\|\rho^{\frac{1}{2}}\dot{u}\|_{L^2}^2\|\nabla u\|_{L^2}^2.
\ea\ee
Together with these estimates of the three boundary terms above, and by \eqref{ax2}, we conclude that
\be\la{ax399}\ba
 J_1& \leq -\left(\int_{\partial\Omega}\sigma^{m}(u\cdot\nabla n\cdot u)Fds\right)_t- (\lambda+2\mu)\int\sigma^{m}(\div\dot{u})^{2}dx \\
&\quad+\frac{\delta}{2}\sigma^{m}\|\nabla\dot{u}\|_{L^2}^2 + C(\delta)\sigma^{m}\|\rho^{\frac{1}{2}}\dot{u}\|_{L^2}^2(\|\nabla u\|_{L^2}^4+1)\\
& \quad+C(\delta)\sigma^{m}(\|\nabla u\|_{L^2}^2+\|\nabla u\|_{L^2}^6+\|\nabla u\|_{L^4}^4).
\ea\ee
By \eqref{curlug}, a direct computation shows that
\be\la{b2}\ba
J_2&=-\mu\int\sigma^{m}\dot{u}\cdot(\nabla\times\curl u_t)dx-\mu\int\sigma^{m}\dot{u}\cdot(\nabla\times\curl u)\,\div udx\\
&\quad-\mu\int\sigma^{m} u^{i}\dot{u}\cdot\nabla\times(\nabla_i\curl u)dx  \\
&=-\mu\int\sigma^{m}|\curl \dot{u}|^{2}dx-\mu\int\sigma^{m}\curl \dot{u}\cdot\curl(u\cdot\nabla u)dx \\
&\quad+\mu\int\sigma^{m}(\curl u\times\dot{u})\cdot\nabla\div udx -\mu\int\sigma^{m}\div u\,\curl u\cdot\curl \dot{u}dx\\
&\quad-\mu\int\sigma^{m} u^{i}\div(\nabla_i\curl u\times\dot{u})dx+\mu\int\sigma^{m} u^i\nabla_i\curl u\cdot\curl\dot{u}dx  \\
&=-\mu\int\sigma^{m}|\curl\dot{u}|^{2}dx-\mu\int\sigma^{m}\curl\dot{u}\cdot(\nabla u^i\times\nabla_i u) dx \\
&\quad+\mu\int\sigma^{m}(\curl u\times\dot{u})\cdot\nabla\div udx-\mu\int\sigma^{m}\div u\,\curl u\cdot\curl \dot{u}dx \\
&\quad-\mu\int\sigma^{m}  u\cdot\nabla\div(\curl u\times\dot{u})dx+\mu\int\sigma^{m} u^i\div(\curl u\times\nabla_i\dot{u})dx \\
&=-\mu\int\sigma^{m}|\curl \dot{u}|^{2}dx+\mu\int\sigma^{m}(\curl u\times\nabla u^i)\cdot\nabla_i\dot{u}dx \\
&\quad-\mu\int\sigma^{m}\curl\dot{u}\cdot(\nabla u^i\times\nabla_i u)dx-\mu\int\sigma^{m}\div u\,\curl u\cdot\curl \dot{u}dx\\
&\leq-\mu\int\sigma^{m}|\curl \dot{u}|^{2}dx+C\sigma^{m}\|\nabla\dot{u}\|_{L^2}\|\nabla u\|_{L^4}^2\\
&\leq-\mu\int\sigma^{m}|\curl \dot{u}|^{2}dx+\frac{\delta}{2}\sigma^{m}\|\nabla\dot{u}\|_{L^2}^2+ C(\delta)\sigma^{m}\|\nabla u\|_{L^4}^4.
\ea\ee
Therefore, we deduce from \eqref{ax1}, \eqref{ax399} and \eqref{b2} that
\be\la{ax40}\ba
&\frac{1}{2}\left(\sigma^{m}\|\rho^{\frac{1}{2}}\dot{u}\|_{L^2}^2\right)_t+(\lambda+2\mu)\sigma^{m}\|\div\dot{u}\|_{L^2}^2+\mu\sigma^{m}\|\curl\dot{u}\|_{L^2}^2\\
& \leq -\left(\int_{\partial\Omega}\sigma^{m}(u\cdot\nabla n\cdot u)Fds\right)_t +\delta\sigma^{m}\|\nabla\dot{u}\|_{L^2}^2\\
&\quad+ C(\delta)\sigma^{m}\|\rho^{\frac{1}{2}}\dot{u}\|_{L^2}^2(\|\nabla u\|_{L^2}^4+1)\\
& \quad+ C(\delta)\sigma^{m}(\|\nabla u\|_{L^2}^2+\|\nabla u\|_{L^2}^6+\|\nabla u\|_{L^4}^4).
\ea\ee

Next, combining  $\eqref{ch1}_1,$ \eqref{bdd2}, and \eqref{bdd1} gives   $$(\dot{u}-(u\cdot\nabla n)\times u^{\perp})\cdot n|_{\partial\Omega}=0,$$ which together with  \eqref{ljq01} yields
\be\la{tb11} \ba\|\nabla\dot{u}\|_{L^2}
&\leq C(\|\div \dot{u}\|_{L^2}+\|\curl \dot{u}\|_{L^2}+\|\nabla[(u\cdot\nabla n)\times u^\perp]\|_{L^2})\\&\leq C(\|\div \dot{u}\|_{L^2}+\|\curl \dot{u}\|_{L^2}+\|\nabla u\|_{L^2}^2+\|\nabla u\|_{L^4}^2),
\ea\ee where in the second inequality we have used
\bnn \ba
 \|\nabla[(u\cdot\nabla n)\times u^\perp]\|_{L^2(\Omega)}
&=\|\nabla[(u\cdot\nabla n)\times u^\perp]\|_{L^2(B_{2R})}\\
&\leq C(R)(\||u||\nabla u|\|_{L^2(B_{2R})}+\|u\|_{L^4(B_{2R})}^2)\\
&\le C(R)(\|\nabla u\|_{L^4(B_{2R})}^2+\|u\|_{L^6(B_{2R})}^2)\\
&\le C(\|\nabla u\|_{L^4}^2+\|u\|_{L^6}^2)\\
&\le C(\|\nabla u\|_{L^4}^2+\|\nabla u\|_{L^2}^2),
\ea\enn
due to  \eqref{ljq10}, \eqref{g1} and H\"{o}lder's inequality.

Now choose $\delta$ small enough, it follows from \eqref{tb11} and \eqref{ax40} that
\be\la{ax401}\ba
&\left(\sigma^{m}\|\rho^{\frac{1}{2}}\dot{u}\|_{L^2}^2\right)_t+(\lambda+2\mu)\sigma^{m}\|\div\dot{u}\|_{L^2}^2+\mu\sigma^{m}\|\curl\dot{u}\|_{L^2}^2\\
& \leq  -\left(2\int_{\partial\Omega}\sigma^{m}(u\cdot\nabla n\cdot u)Fds\right)_t +C\sigma^{m}\|\rho^{\frac{1}{2}}\dot{u}\|_{L^2}^2(\|\nabla u\|_{L^2}^4+1)\\
&\quad+C\sigma^{m}(\|\nabla u\|_{L^2}^2+\|\nabla u\|_{L^2}^6+\|\nabla u\|_{L^4}^4).
\ea\ee

Taking $m=3$ in \eqref{ax401}, and using \eqref{tb11}, \eqref{ljq11}, \eqref{zz1} and Lemma \ref{le3}, we  establish \eqref{h15} and complete the proof of Lemma \ref{xcrle1}.
\end{proof}

\begin{lemma}\la{zc1} If $(\n,u)$ be a smooth solution of \eqref{a1}-\eqref{ch2} on $\O \times (0,T] $ satisfying \eqref{zz1} and $\|\nabla u_0\|_{L^2}\leq M$, then there exist two positive constants $C=C(\on,M)$ and $\varepsilon_1$ depending only on $\mu,\,\,\lambda,\,\,\gamma,\,\,a,\,\,\rho_\infty,\,\,\bar{\rho},\,\,\Omega$ and $M,$ such that
   \be\la{uv1}  \sup_{0\le t\le \si(T)}\|\na
u\|_{L^2}^2+\int_0^{\si(T)}\int\n|\dot u|^2dxdt\le
C(\on,M), \ee
 \be\la{uv2}  \sup_{0\le t\le \si(T)}t\int\n|\dot u|^2dx+\int_0^{\si(T)}t\int|\nabla\dot{u}|^2dxdt\le
C(\on,M). \ee
\end{lemma}

\begin{proof} Multiplying $(\ref{a1})_2 $ by $u_t$ and integrating over $\Omega$, a direct calculation gives
\be \la{xbh13} \ba
&\left(\frac{\lambda+2\mu}{2}\int(\div u)^{2}dx+\frac{\mu}{2}\int|\curl u|^{2}dx-\int(P-P(\rho_\infty))\div udx\right)_t+\int\rho|\dot u|^{2}dx \\
& =\int\rho\dot u\cdot(u\cdot\nabla u)dx-\int P_t\div udx \\
& =\int\rho\dot u\cdot(u\cdot\nabla u)dx - \frac{1}{\lambda+2\mu}\int(P-P(\rho_\infty))F\div u dx \\
&\quad- \frac{1}{\lambda+2\mu}\int(P-P(\rho_\infty))\nabla F\cdot u dx- \frac{1}{2(\lambda+2\mu)}\int(P-P(\rho_\infty))^{2}\div u dx\\
& \quad  +\gamma\int P(\div u)^2dx \\
& \leq C(\bar{\rho})(\|\rho^{\frac{1}{2}}\dot u\|_{L^2}\|\rho^{\frac{1}{3}}u\|_{L^3}\|\nabla u\|_{L^6}+\|P-P(\rho_\infty)\|_{L^3}\|\nabla F\|_{L^2}\|u\|_{L^6})\\
&\quad+C(\bar{\rho})(\|\nabla u\|_{L^2}\|F\|_{L^2}+\|P-P(\rho_\infty)\|_{L^2}\|\nabla u\|_{L^2}+\|\nabla u\|_{L^2}^2)\\
&\leq (C(\bar{\rho})C_0^{\frac{1}{12}}+\frac{1}{4})\|\rho^{\frac{1}{2}}\dot u\|_{L^2}^{2}+C(\bar{\rho})(\|\nabla u\|_{L^2}^2+\|P-P(\rho_\infty)\|_{L^2}^2)\\
&\quad+C(\bar{\rho})\|P-P(\rho_\infty)\|_{L^6}^2,
\ea\ee
where we have taken advantage of \eqref{zz1}, \eqref{h19}, \eqref{h20}, \eqref{h18}, H\"{o}lder's, Poincar\'{e}'s and Young's inequalities.

Hence,
\be \la{xbh14} \ba
&\left((\lambda+2\mu)\|\div u\|_{L^{2}}^{2}+\mu\|\curl u\|_{L^{2}}^{2}-2\int(P-\bar{P})\div udx\right)_t
+\int\rho|\dot u|^{2}dx\\
& \le C(\bar{\rho})\left(\|\nabla u\|_{L^{2}}^{2}+\|P-P(\rho_\infty)\|_{L^{2}}^{2}+\|P-P(\rho_\infty)\|_{L^6}^2\right) ,
\ea\ee
provide that $C_0<\varepsilon_1\triangleq(4C(\bar{\rho}))^{-12}$.

Therefore, it follows from Gronwall's inequality, Lemmas \ref{crle1} and \ref{le2}  that \eqref{uv1} holds.

Finally, we will claim \eqref{uv2}.  Taking $m=1$ in \eqref{ax401}, and integrating over $(0,\sigma(T)]$, we deduce from \eqref{uv1} that
\be \la{xbh16} \ba
&\sup_{0\le t\le  \si(T)}t\|\rho^{\frac{1}{2}}\dot{u}\|_{L^2}^2+\int_0^{\sigma(T)}t\|\nabla\dot{u}\|_{L^2}^2dt\\
& \leq C\int_0^{\sigma(T)}\|\rho^{\frac{1}{2}}\dot{u}\|_{L^2}^2dt
+C\int_0^{\sigma(T)}t\|\rho^{\frac{1}{2}}\dot{u}\|_{L^2}^2(\|\nabla u\|_{L^2}^4+1)dt\\
&\quad +C\int_0^{\sigma(T)}t(\|\nabla u\|_{L^2}^2+\|\nabla u\|_{L^2}^6)dt+C\int_0^{\sigma(T)}t\|\nabla u\|_{L^4}^4dt\\
& \quad +C\int_0^{\sigma(T)}\|\nabla u\|_{L^2}^4dt+Ct(\|\nabla u\|_{L^2}^2+\|\nabla u\|_{L^2}^4)\\
& \le  C\int_0^{\si(T)}t\|\na u\|_{L^4}^4dt+C(\bar{\rho}, M). \ea\ee
On the other hand, by \eqref{h18} and \eqref{uv1},
\be \la{xbh17}\ba
&\int_0^{\sigma(T)}t\|\nabla u\|_{L^{4}}^{4}dt \\
& \le C\int_0^{\sigma(T)}t\left(\|\rho\dot{u}\|_{L^2}^3\|\nabla u\|_{L^2}+\|P-P(\rho_\infty)\|_{L^6}^3\|\nabla u\|_{L^2}+\|\nabla u\|_{L^2}^4\right)dt\\
& \le C\int_0^{\sigma(T)}t^{\frac{1}{2}}(\|\nabla u\|_{L^2}^2)^{\frac{1}{2}}(t\|\rho^{\frac{1}{2}}\dot{u}\|_{L^{2}}^{2})^{\frac{1}{2}}(\|\rho^{\frac{1}{2}}\dot{u}\|_{L^2}^2)dt+C(\bar{\rho})\\
& \leq C(\bar{\rho}, M)\left(\sup_{0\le t\le  \si(T)}t\|\rho^{\frac{1}{2}}\dot{u}\|_{L^2}^2\right)^{\frac{1}{2}}+C(\bar{\rho}) .
\ea\ee
Therefore, together with \eqref{xbh16} and \eqref{xbh17}, we get \eqref{uv2}.
\end{proof}
\begin{lemma}\la{nzc1} Suppose $(\n,u)$ is a smooth solution of \eqref{a1}-\eqref{ch2}   on $\O \times (0,T] $ satisfying \eqref{zz1} and the assumption $\|\nabla u_0\|_{L^2}\leq M$, then there
exists a positive constant  $\varepsilon_2\,(\leq\varepsilon_1)$ depending only on $\mu ,  \lambda ,   \ga ,  a ,  \on, \rho_\infty,$ $\Omega$ and $M$ such
that
\be\la{xuv1} \sup_{0\le t\le  \si(T) }\int \n |u|^{3}dx\le C_0^{\frac{1}{4}} ,\ee
provided $C_0<\varepsilon_2$.
\end{lemma}

\begin{proof}
%For $s\in (0,1],$   Sobolev's inequality implies
%\be \la{zz.1}\ba\int \n_0|u_0|^{2+\nu}dx &\le  \int \n_0|u_0|^{2 }dx +  \int \n_0 |u_0|^{2/(1-s)}dx\\ &\le C(\on)+C(\on)\|u_0\|_{\dot %H^\beta}^{2/(1-s)}\le C(\on,M).\ea\ee For the case that $s=1,$ one obtains from \eqref{a16} that
%\be  \la{zz.2}\ba\int \n_0|u_0|^{2+\nu}dx &\le C(\on)\left( \int \n_0|u_0|^{2 }dx +  \int   |\na u_0|^2dx\right)^{(2+\nu)/2} \le C(\on,M).\ea\ee
Multiplying $(\ref{a1})_2$ by $3|u|u$, integrating over $ \O$, by Lemma \ref{le2}, \eqref{h18} and \eqref{uv1}, we check that
\bnn\ba
&\left(\int\rho|u|^{3}dx\right)_t=-3(\lambda+2\mu)\int\div u\,\div(|u|u)dx-3\mu\int\curl u\cdot\curl(|u|u)dx  \\
& \quad + 3\int(P-P(\rho_\infty))\div(|u|u)dx\\
&\leq C\int|u||\nabla u|^{2}dx+C\int|P-P(\rho_\infty)||u||\nabla u|dx\\
&\leq C\|u\|_{L^6}\|\nabla u\|_{L^2}^{\frac{3}{2}}\|\nabla u\|_{L^6}^{\frac{1}{2}}+C\|P-P(\rho_\infty)\|_{L^3}\|u\|_{L^6}\|\nabla u\|_{L^2}\\
&\leq C\|\nabla u\|_{L^2}^{\frac{5}{2}}(\|\rho\dot{u}\|_{L^2}+\|P-P(\rho_\infty)\|_{L^6}+\|\nabla u\|_{L^2})^{\frac{1}{2}}+C(\bar{\rho})C_0^{\frac{1}{6}}\|\nabla u\|_{L^2}^2\\
&\leq C(\bar{\rho})(\|\nabla u\|_{L^2}^2)^{3/4}(\|\rho^{\frac{1}{2}}\dot{u}\|_{L^2}^2)^{\frac{1}{4}}\|\nabla u\|_{L^2}
+CC_0^{\frac{1}{12}}(\|\nabla u\|_{L^2}^2)^{\frac{3}{4}}\|\nabla u\|_{L^2}\\
&\quad+C(\|\nabla u\|_{L^2}^2)\|\nabla u\|_{L^2}+C(\bar{\rho})C_0^{\frac{1}{6}}\|\nabla u\|_{L^2}^2,
\ea\enn
which along with \eqref{a16} and \eqref{uv1} yields
\be\la{bcbh1} \ba
&\sup_{0\le t\le  \si(T) }\int\rho|u|^{3}dx\\
&\leq C(\bar{\rho},M)\left(\int_0^{\sigma(T)}\|\nabla u\|_{L^2}^2dt\right)^{\frac{1}{2}}
+\int\rho_0|u_0|^3dx+CC_0+CC_0^{\frac{1}{2}}\\
&\leq C(\bar{\rho},M)C_0^{\frac{1}{2}},
\ea\ee
provided $C_0<\varepsilon_1,$ where in the last inequality we have used the simple fact
\be\la{bcbh2} \ba
 \int\rho_0|u_0|^{3}dx\leq C(\bar{\rho})\|\rho_0^{\frac{1}{2}}u_0\|_{L^{2}}^{3/2}\|\nabla u_0\|_{L^2}^{3/2}\leq C(\bar{\rho},M)C_0^{1/2}.
\ea\ee

 Now let $\varepsilon_2\triangleq\min\{\varepsilon_1,(C(\bar{\rho}, M))^{-4}\}$, and we establish \eqref{xuv1}.
\end{proof}

\begin{lemma}\la{xle5}
      If  $(\rho,u)$ is a smooth solution  of
   \eqref{a1}-\eqref{ch2}     on $\O \times (0,T] $ satisfying \eqref{zz1}, then
  \be\la{lx61} \ba
&\int_0^{T} \sigma^3\|P-P(\rho_\infty)\|_{L^4}^4 dt \le  C(\on)C_0. \ea \ee
   \end{lemma}
\begin{proof}
It follows from $(\ref{a1})_1$ that $P-P(\rho_\infty)$ satisfies
\be\la{lx4}(P-P(\rho_\infty))_t+u\cdot\nabla (P-P(\rho_\infty))+\ga(P-P(\rho_\infty)){\rm div}u+\ga P(\rho_\infty){\div}u=0.\ee
 Multiplying (\ref{lx4}) by $3 (P-P(\rho_\infty))^2$ and integrating over $\Omega,$ after using $\div u=\frac{1}{2\mu+\lambda}(F+P-P(\rho_\infty))$ and \eqref{h20}, we get
\be \la{lx5}\ba
& \frac{3\ga-1}{2\mu+\lambda}\|P-P(\rho_\infty)\|_{L^4}^4  \\
&=-\left(\int(P-P(\rho_\infty))^3dx\right)_t-\frac{3\ga-1}{2\mu+\lambda}\int(P-P(\rho_\infty))^3Fdx\\
&\quad-3\ga P(\rho_\infty)\int (P-P(\rho_\infty))^2{\rm div}udx\\
&\le-\left(\|P-P(\rho_\infty)\|_{L^3}^3\right)_t+\delta\|P-P(\rho_\infty)\|_{L^4}^4+C(\delta)\|F\|_{L^4}^4+C(\delta)\|\nabla u\|_{L^2}^2\\
&\le-\left(\|P-P(\rho_\infty)\|_{L^3}^3\right)_t+\delta\|P-P(\rho_\infty)\|_{L^4}^4+C(\delta)\|\rho\dot{u}\|_{L^2}^3\|\nabla u\|_{L^2}\\
&\quad+C(\delta)\|\rho\dot{u}\|_{L^2}^3\|P-P(\rho_\infty)\|_{L^2}+C(\delta)\|\nabla u\|_{L^2}^2.\ea\ee
Multiplying (\ref{lx5}) by $\si^3$, then integrating over $(0,T],$ and choosing $\delta$ suitably small, by \eqref{zz1}, we obtain
\bnn \ba
&\int_0^{T} \sigma^3\|P-P(\rho_\infty)\|_{L^4}^4 dt \\
& \le C\sup_{0\le t\le T}\|P-P(\rho_\infty)\|^3_{L^3}+C\int_0^{\si(T)}\|P-P(\rho_\infty)\|^3_{L^3}dt\\
&\quad+C(\on)\int_0^{T} \sigma^3\|\rho\dot{u}\|_{L^2}^3\|\nabla u\|_{L^2}ds+C\int_0^{T}\sigma^3\|\rho\dot{u}\|_{L^2}^3\|P-P(\rho_\infty)\|_{L^2} ds +C(\on) C_0 \\
&\le C(\on) \sup_{t\in (0,T]}\left( \si^{3/2}\|\rho^{\frac{1}{2}} \dot u\|_{L^2}\right) \left(\si^{1/2}\|\nabla u\|_{L^2}\right)
\int_0^{T}\sigma\|\rho^{\frac{1}{2}}\dot{u}\|_{L^2}^2ds\\
&\quad+C\sup_{t\in (0,T]}C_0^{\frac{1}{2}}\left( \si^3\|\rho^{\frac{1}{2}} \dot u\|_{L^2}^2\right)^{\frac{1}{2}}\int_0^{T}\sigma^{\frac{1}{2}}\sigma\|\rho^{\frac{1}{2}}\dot{u}\|_{L^2}^2ds+C(\on) C_0\\
&\le  C(\on)C_0.\ea \enn
This completes the proof.
\end{proof}
\begin{lemma}\la{le5}
      There exists a positive constant
    $\ve_3 (\leq\ve_2)$ depending only on $\mu ,  \lambda ,   \ga ,  a ,  \on$, $\rho_\infty$, $\Omega$ and $M$ such that,  if  $(\rho,u)$ is a smooth solution  of
   \eqref{a1}-\eqref{ch2} on $\O \times (0,T] $ satisfying \eqref{zz1} and $\|\nabla u_0\|_{L^2}\leq M$, then
  \be\la{lx1}
  A_1(T)+A_2(T)\le C_0^{1/2},
  \ee provided $C_0\le \ve_3.$
   \end{lemma}
\begin{proof} By Lemma \ref{xcrle1}, it indicates that
\be\la{lx11}\ba
A_1(T)+A_2(T)\le C(\on)C_0 + C(\on)\int_0^{T}     \sigma^3\|\nabla
u\|_{L^4}^4 ds+C(\on)\int_0^{T} \sigma\|\nabla u\|_{L^3}^3ds. \ea\ee
Due to \eqref{h18}, \eqref{h20}, \eqref{zh20} and Lemma \ref{xle5},
\be\la{lx7} \ba
&  \int_0^{T} \sigma^3\|\nabla u\|_{L^4}^4 ds\\
& \le C\int_0^{T}\sigma^3\|\nabla u\|_{L^2}\|\rho\dot{u}\|_{L^2}^3ds+C\int_0^{T}\sigma^3\|\nabla u\|_{L^2}^2ds\\
&\quad+C\int_0^{T}\sigma^3\|P-P(\rho_\infty)\|_{L^6}^6 ds
+C\int_0^{T}\sigma^3\|\nabla u\|_{L^2}^4ds\\
& \le C(\on) \sup_{t\in (0,T]}\left( \si^{3/2}\|\rho^{\frac{1}{2}} \dot u\|_{L^2}\right) \left(\si^{1/2}\|\nabla u\|_{L^2}\right)
\int_0^{T}\int\sigma\rho|\dot{u}|^2dxds\\
&\quad+C(\on)\int_0^{T}\sigma^3\|P-P(\rho_\infty)\|_{L^4}^4 ds+C\int_0^{T}(\sigma\|\nabla u\|_{L^2}^2)\sigma^2\|\nabla u\|_{L^2}^2ds\\
&\le C(\on) A_2^{1/2}(T)A_1^{1/2}(T)A_1(T)+CC_0\\
&\le C(\on)C_0.  \ea \ee

Finally, we shall give the estimate of  the last term on the right hand side of
(\ref{lx11}). By (\ref{lx7}), it is easy to check that
\be\la{lx8} \ba
&\int_{\si(T)}^{T}\int\sigma|\nabla u|^3dxds \le
\int_{\si(T)}^{T}\int( |\nabla u|^4 + |\nabla u|^2)dxds  \le  C C_0.
\ea \ee Next, we deduce from (\ref{h18}), (\ref{uv1}) and
(\ref{zz1})  that if $C_0\le
\ve_1$, then
\be\la{lx9} \ba
& \int_0^{\si(T)} \sigma\|\nabla u\|_{L^3}^3 dt\\
& \le  C(\on)\int_0^{\si(T)}t  \|\nabla u\|_{L^2}^{3/2} \left(\|\rho\dot{u}\|^{3/2}_{L^2}+C_0^{1/4}\right)dt+C\int_0^{\si(T)}\sigma\|\nabla u\|_{L^2}^3 dt\\
& \le C(\on)\int_0^{\si(T)}\left( \|\nabla u\|_{L^2}\right)\|\nabla u\|_{L^2}^{1/2}   \left(t\int\rho|\dot{u}|^2dx\right)^{3/4}dt+ C(\on)C_0
\\  & \le C(\on)\sup_{t\in (0,\si(T)]} (\|\nabla u\|_{L^2})     \int_0^{\si(T)}  \|\nabla u\|_{L^2}^{1/2}
\left(t \int\rho|\dot{u}|^2dx\right)^{3/4}dt +C(\on)C_0\\
&\le C(\bar \n,M )A_1^{3/4}C_0^{1/4}  +
C(\on)C_0\\
&\le C(\bar \n,M) C_0^{5/8}. \ea \ee
Taking
\bnn
\ve_3\triangleq\min\left\{\ve_2,\left(C(\bar \n,M)\right)^{-8}
\right\},\enn
By (\ref{lx11}) and (\ref{lx7})-(\ref{lx9}), we immediately obtain
(\ref{lx1}) in the case $C_0<\ve_3$.
\end{proof}

It is time to derive a uniform (in time) upper bound for the
density, which plays a crucial role in deriving all the higher
order estimates and thus extending the local classical solution to be a global one.
The method we are going to employ can be found in \cite{lx} \cite{jx01}.
\begin{lemma}\la{le7} Let $(\rho,u)$ be a smooth solution  of
(\ref{a1})--(\ref{ch2}) on $\Omega\times (0,T] $
satisfying (\ref{zz1}) and $\|\nabla u_0\|_{L^2}\leq M$, then
there exists a positive constant
   $\ve$ as described in Theorem \ref{th1} depending only on $\mu ,  \lambda ,   \ga ,  a ,  \on, \rho_\infty$, $\Omega$ and $M$ such that,
      \bnn\sup_{0\le t\le T}\|\n(t)\|_{L^\infty}  \le
\frac{7\bar \n }{4}  ,\enn
      provided $C_0\le \ve . $

   \end{lemma}

\begin{proof} Denote $$
D_t\n\triangleq\n_t+u \cdot\nabla \n ,\quad
g(\n)\triangleq-\frac{\rho(P-P(\rho_s))}{2\mu+\lambda}  ,
\quad b(t)\triangleq-\frac{1}{2\mu+\lambda} \int_0^t\n Fdt, $$
then $(\ref{a1})_1$ can be rewritten as
 \be \la{z.3} D_t \n=g(\n)+b'(t). \ee
In order to finish the proof, by Lemma \ref{le1}, it is sufficient to check that the function $b(t)$ must verify \eqref{a100} with some suitable constants $N_0$, $N_1$.

Let $\ve_3$ be given in Lemma \ref{le5}. One deduces from Gagliardo-Nirenberg's inequality, \eqref{zz1}, \eqref{uv1} and \eqref{uv2} that  for all $ 0\le t_1<t_2\le \si(T),$ if $C_0\le\ve_3,$ then
\be\ba &|b(t_2)-b(t_1)|\\&\le C\int_0^{\si(T)}\|(\n
F)(\cdot,t)\|_{L^\infty}dt\\ &\le C(\on)
\int_0^{\si(T)}\|F(\cdot,t)\|^{1/2}_{L^6}\|\na
F(\cdot,t)\|^{1/2}_{L^6}dt\no&\le  C(\on) \int_0^{\si(T)}
\|\n \dot u\|^{1/2}_{L^2} \|\na\dot u\|^{1/2}_{L^2}dt \no&\le
C(\on)\int_0^{\si(T)}t^{-\frac{5}{16}}( t\|\n\dot
u\|_{L^2}^2)^{\frac{1}{16} }\|\n\dot
u\|_{L^2}^{\frac{3}{8}}\left(t\|\na\dot u\|_{L^2}^2\right)^{\frac14}dt
\no&\le
C(\on,M)\left(\int_0^{\si(T)}t^{-\frac{2}{3}}
 \left(t\|\n^{1/2}\dot
u\|_{L^2}^2\right)^{\frac{1}{4}} dt\right)^{3/4} \no&\le C(\on,M)
(A_1(\si(T)))^{\frac{3}{16}}\no&\le C(\on,M)C_0^{\frac{3}{32}}.\ea\ee
Therefore, for $ t\in [0,\si(T)],$ one can
choose $N_0$ and $N_1$ in (\ref{a100}) as follows:\bnn N_1=0, \quad
N_0=C(\on,M)C_0^{\frac{3}{32}}, \enn and $\bar\zeta=\rho_\infty$ in
(\ref{a101}). It is easy to check that
$$ g(\zeta)=-\frac{a\zeta}{2\mu+\lambda}(\zeta^{\ga}-\rho_\infty^{\ga})
\le -N_1=0, \quad \mbox{for all}\quad { \zeta}\ge \bar\zeta= \rho_\infty.$$
So Lemma \ref{le1} yields that \be\la{lx20}\sup_{t\in
[0,\si(T)]}\|\rho\|_{L^\infty}\le \max\{  \bar\n ,\rho_\infty\}+N_0\le \on
+C(\on,M)C_0^{\frac{3}{32}}\le\frac{3 \bar\n  }{2},\ee
 provided $$C_0\le\min\{\ve_3,\ve_4\},\quad\mbox{ for }
 \ve_4\triangleq\left(\frac{\bar \n }{2C(\on,M) }\right)^{\frac{32}{3}}.$$

On the other hand, by Gagliardo-Nirenberg's inequality, \eqref{zz1} and \eqref{h19},
 for all $\si(T)\le t_1\le t_2\le T,$ \be\ba  |b(t_2)-b(t_1)|
 &\le  C(\on) \int_{t_1}^{t_2}\|F(\cdot,t)\|_{L^\infty}dt \no& \le
\frac{a}{2\mu+\lambda}(t_2-t_1)+C(\on)
\int_{\si(T)}^{T}\|F(\cdot,t)\|^{4}_{L^\infty}dt\no & \le \frac{
a}{2\mu+\lambda}(t_2-t_1)+C(\on)
\int_{\si(T)}^T\|F(\cdot,t)\|^{2}_{L^6}\|\nabla
F(\cdot,t)\|^{2}_{L^6}dt \no& \le \frac{a}{2\mu+\lambda}(t_2-t_1)+ C
(\bar{\rho})\int_{\si(T)}^T \|\rho^{\frac{1}{2}}\dot{u}\|^{2}_{L^2}\|\na \dot
u(\cdot,t)\|^{2}_{L^2}dt\no& \le \frac{a}{2\mu+\lambda}(t_2-t_1)+ C
(\bar\n)C_0,\ea\ee
 Therefore, one can choose $N_1$ and $N_0$ in (\ref{a100}) as: \bnn
N_1=\frac{a}{2\mu+\lambda}, \quad N_0=C(\on)C_0.\enn Note that
$$ g(\zeta)=-\frac{a\zeta}{2\mu+\lambda}(\zeta^{\ga}-\rho_\infty^\ga)\le -N_1
=-\frac{ a}{2\mu+\lambda}, \quad \mbox{for all}\quad { \zeta}\ge
\rho_\infty+1.$$ So one can set $\bar\zeta=\rho_\infty+1 $ in (\ref{a101}). By Lemma
\ref{le1} and (\ref{lx20}), we get
\be\la{lx21}
\sup_{t\in[\si(T),T]}\|\rho\|_{L^\infty}\le \max\left\{ \frac{ 3\bar \n
}{2},\rho_\infty+1\right\}+N_0\le \frac{3\bar \n }{2 } +C(\on)C_0\le
\frac{7\bar \n }{4} ,\ee provided
\be\la{lx30} C_0\le
\ve\triangleq\min\{\ve_3,\ve_4,\ve_5\}, \quad\mbox{ for
}\ve_5\triangleq \frac{\bar{\rho}}{4C(\bar{\rho})}.\ee
Combining  (\ref{lx20}) and (\ref{lx21}), we finish the
proof of Lemma \ref{le7}.
\end{proof}
\section{\la{se5} A priori estimates (II): higher order estimates }

From now on, we assume that the initial energy $C_0$  always satisfies  (\ref{lx30}) and the positive
constant $C $ may depend on \bnn  T,\,\, \| g\|_{L^2},  \,\,\|\na u_0\|_{H^1},\,\,
    \|\n_0-\rho_\infty\|_{H^2\cap W^{2,q}}  ,  \,\, \|P(\n_0)-P(\rho_\infty)\|_{H^2\cap W^{2,q}} , \,\,\enn
besides $\mu$, $\lambda$, $\rho_\infty$, $a$, $\ga$, $\on,$  $\Omega$ and $
M,$ where $g\in L^2(\Omega)$ is given by the compatibility condition
\eqref{dt3}.
This section is devoted to some necessary higher order estimates of
the smooth solution $(\rho,u)$ which will make sure that the local existence of classical solution can be extended globally in time. The methods used are mainly from \cite{jx01}. For completeness, we still write it down in detail.
\begin{lemma}\la{xle1}
 There exists a positive constant $C,$ such that
\be\label{cxb2}
\sup_{0\le t\le T}\|\rho^{\frac{1}{2}}\dot{u}\|_{L^2}+\int_0^T\|\nabla\dot{u}\|_{L^2}^{2}dt\leq C,\ee
\be\label{cxb3}
\sup_{0\le t\le T}(\|\nabla\rho\|_{{L^2}\cap{L^6}}+\|\nabla u\|_{H^1})+\int_0^T\|\nabla u\|_{L^\infty}dt\leq C.\ee

\end{lemma}
\begin{proof} First, by (\ref{uv1}) and
(\ref{lx1}), we obtain
     \be\la{cxb4}
  \sup_{t\in[0,T]}\|\nabla u\|_{L^2}^2 + \int_0^{T}\int\rho|\dot{u}|^2dxdt
  \le C .
  \ee
Choosing $m=0$ in \eqref{ax401}, we deduce from \eqref{h18}, \eqref{a16} and \eqref{cxb4} that
\be\la{cxb5} \ba
\left(\|\rho^{\frac{1}{2}}\dot{u}\|_{L^2}^2\right)_t+\|\nabla\dot{u}\|_{L^2}^2
\le C(\|\rho^{\frac{1}{2}}\dot{u}\|_{L^2}^2+1)\|\rho^{\frac{1}{2}}\dot{u}\|_{L^2}^2+C-\left(\int_{\partial\Omega}(u\cdot\nabla n\cdot u)Fds\right)_t.\ea\ee
As a result, by Gronwall's inequality, together with the compatibility condition \eqref{dt3}, \eqref{cxb4}, \eqref{cxb5} and \eqref{ljq11}, we give \eqref{cxb2}.

Observe that for $2\leq p\leq 6 ,$ $|\nabla\rho|^p$ satisfies \bnnn \ba
& (|\nabla\rho|^p)_t + \text{div}(|\nabla\rho|^pu)+ (p-1)|\nabla\rho|^p\text{div}u  \\
 &+ p|\nabla\rho|^{p-2}(\nabla\rho)^{\ast} \nabla u (\nabla\rho) +
p\rho|\nabla\rho|^{p-2}\nabla\rho\cdot\nabla\text{div}u = 0.\ea
\ennn
Thus, due to \eqref{h19},
\be\la{cxb9}\ba
(\|\nabla\rho\|_{L^p})_t&\le C(1+\norm[L^{\infty}]{\nabla u} )\norm[L^p]{\nabla\rho} +C\|\na F\|_{L^p}\\
&\le C(1+\norm[L^{\infty}]{\nabla u} )\norm[L^p]{\nabla\rho}+C\|\rho\dot{u}\|_{L^p}. \ea\ee

On the other hand, by Gagliardo-Nirenberg's inequality, we deduce from \eqref{a16}, \eqref{cxb4}, \eqref{h19} and \eqref{zh19} that
\be\la{cxb11}\ba
&\|\div u\|_{L^\infty}+\|\curl u\|_{L^\infty}\\
&\le C(\|F\|_{L^\infty}+\|P-P(\rho_\infty)\|_{L^\infty})+\|\curl u\|_{L^\infty} \\
&\le C(\|F\|_{L^2}+\|\nabla F\|_{L^6}+\|\curl u\|_{L^2}+\|\nabla \curl u\|_{L^6}+\|P-P(\rho_\infty)\|_{L^\infty}) \\
&\le C(\|\nabla u\|_{L^2}+\|P-P(\rho_\infty)\|_{L^2}+\|\rho\dot{u}\|_{L^6}+\|P-P(\rho_\infty)\|_{L^\infty}) \\
&\le C(\bar{\rho})(\|\nabla\dot{u}\|_{L^2}+1).
\ea\ee Moreover, by Lemma \ref{crle5} and \eqref{x266}-\eqref{x268}, we have for $p\in[2,6]$,
\be\la{remark1}\ba
\|\nabla^{2}u\|_{L^p}&\leq C(\|\div u\|_{W^{1,p}}+\|\curl u\|_{W^{1,p}}+\|\nabla u\|_{L^2})\\
&\leq C(\|\rho\dot{u}\|_{L^p}+\|\rho\dot{u}\|_{L^2}+\|P-P(\rho_\infty)\|_{W^{1,p}})\\& \quad+C(\|\nabla u\|_{L^2}+\|P-P(\rho_\infty)\|_{L^2}),
\ea\ee
which, along with Lemma \ref{le9}, \eqref{cxb2}, and \eqref{cxb11} yields
\be\la{cxb12}\ba
\|\na u\|_{L^\infty } &\le C\left(\|{\rm div}u\|_{L^\infty }+
\|\curl u\|_{L^\infty } \right)\log(e+\|\na^2 u\|_{L^6 }) +C\|\na
u\|_{L^2} +C \\
&\le C(1+\|\nabla\dot{u}\|_{L^2})\log(e+\|\nabla\dot u\|_{L^2 } +\|\na \rho\|_{L^6}) \\
&\le C(1+\|\nabla\dot{u}\|_{L^2})\left[\log(e+\|\nabla\dot u\|_{L^2 }) +\log(e+\|\nabla\rho\|_{L^6})\right] \\
&\le C(1+\|\nabla\dot{u}\|_{L^2}^2)+C(1+\|\nabla\dot{u}\|_{L^2})\log(e+\|\nabla\rho\|_{L^6}) .
\ea\ee
Hence,
\be\la{cxb14}\ba
\left(\log(e+\|\nabla\rho\|_{L^6})\right)_t\leq C(\bar{\rho})(1+\|\nabla\dot{u}\|_{L^2}^{2})+C(\bar{\rho})(1+\|\nabla\dot{u}\|_{L^2})\log(e+\|\nabla\rho\|_{L^6}).
\ea\ee
And then by Gronwall's inequality and \eqref{cxb2}, we find that
\be\la{cxb15}\ba
\sup_{0\leq t\leq T}\|\nabla\rho\|_{L^{6}}\leq C .
\ea\ee
Moreover, by \eqref{cxb12},
\be\la{cxb160}\ba
\int_0^T\|\nabla u\|_{L^\infty}dt\leq C.
\ea\ee
Taking $p=2$ in \eqref{cxb9}, by  Gronwall's inequality, together with  \eqref{cxb4} and \eqref{cxb160}, we check that
\be\la{cxb161}\ba
\sup_{0\leq t\leq T}\|\nabla\rho\|_{L^{2}}\leq C.
\ea\ee

Furthermore, by \eqref{cxb2}, \eqref{cxb4}, \eqref{cxb15}, \eqref{cxb161} and \eqref{remark1},
\be\la{cxb16}\ba
\int_0^T\|\nabla^{2} u\|_{L^6}^{2}dt\leq C,\,\, \sup_{0\leq t\leq T}\|\nabla^{2}u\|_{L^{2}}\leq C .
\ea\ee
So we finish the proof of Lemma \ref{xle1}.
\end{proof}
\begin{lemma}\la{xle2}
 There exists a constant $C$ such that
\be\la{cxb17}\ba
\sup_{0\le t\le T}\|\rho^{\frac{1}{2}}u_t\|_{L^2}^2 + \ia\int|\nabla u_t|^2dxdt\le C,
\ea\ee
\be\la{cxb18}\ba
\sup_{0\le t\le T}(\|{\rho- \rho_\infty}\|_{H^2} +
 \|{P- P(\rho_\infty)}\|_{H^2})\le C.
\ea\ee
\end{lemma}
\begin{proof} By Lemma \ref{xle1},  we have
\be\la{cxb19}\ba
\|\rho^{\frac{1}{2}}
u_t\|_{L^2}^2 &\le \|\rho^{\frac{1}{2}}\dot u \|_{L^2}^2+\|\n^{\frac{1}{2}}u\cdot\na u\|_{L^2}^2\\
&\le C+C\|\rho^{\frac{1}{2}}u\|_{L^2}\|u\|_{L^6}\|\nabla u\|_{L^6}^2 \\
&\le C ,
\ea\ee
and
\be \la{cxb20}\ba
 \int_0^T\|\nabla u_t\|_{L^2}^2dt &\le\int_0^T\|\nabla \dot
u\|_{L^2}^2dt + \int_0^T\|\nabla(u\cdot\nabla u)\|_{L^2}^2dt \\
&\le C+\int_0^T\|\nabla u\|_{L^4}^4+\|u\|_{L^\infty}^2\|\nabla^{2}u\|_{L^2}^2dt  \\
&\le C+C\int_0^T(\|\nabla^{2}u\|_{L^2}^4+\|\nabla u\|_{H^1}^{2}\|\nabla^{2}u\|_{L^2}^2)dt \\
&\le C.
\ea\ee
which give \eqref{cxb17}.

 Notice that $P $ satisfies
 \bn \la{cxb21} P_t+u\cdot\nabla P+\ga P{\rm div}u=0,
\en
and by $(\ref{a1})_1$, Lemma \ref{xle1}, \eqref{remark1} and \eqref{remark2}, one has
\be \la{cxb22}\ba
&\frac{d}{dt}\left(\|\nabla^2P\|_{L^2}^2 +\|\nabla^2\rho\|_{L^2}^2\right)\\
&\le C(1+\|\nabla u\|_{L^\infty})(\|\nabla^2P\|_{L^2}^2 +\|\nabla^2\rho
\|_{L^2}^2) + C\|\nabla\dot{u}\|_{L^2}^2 + C ,
\ea\ee where we have used the following simple fact that   for $p\in[2,6]$,
\be\la{remark2}\ba
&\|\nabla^{3}u\|_{L^p}\leq C(\|\div u\|_{W^{2,p}}+\|\curl u\|_{W^{2,p}})\\
&\leq C(\|\rho\dot{u}\|_{W^{1,p}}+\|P-P(\rho_\infty)\|_{W^{2,p}}+\|\rho\dot{u}\|_{L^2}+\|\nabla u\|_{L^2}+\|P-P(\rho_\infty)\|_{L^2}),
\ea\ee due to Lemma \ref{crle5} and \eqref{x266}-\eqref{x268}.
Then, by Gronwall's inequality, it follows from \eqref{cxb22}, \eqref{cxb160} and Lemma \ref{xle1} that
  \bnn \sup_{0\le t\le T} {\left(\|\nabla^2P\|_{L^2}^2
+\|\nabla^2\rho\|_{L^2}^2 \right)}\le C. \enn And the proof is completed.
\end{proof}
\begin{lemma}\la{xle3}
There exists a positive constant $C$ such that \be\la{cxb24}
   \sup\limits_{0\le t\le T}\left(
   \|\n_t\|_{H^1}+\|P_t\|_{H^1}\right)
    + \int_0^T\left(\|\n_{tt}\|_{L^2}^2+\|P_{tt}\|_{L^2}^2\right)dt
\le C,
  \ee
%\be\la{nq1}  \sup\limits_{0\le t\le T}\si\|\nabla u_t\|_{L^2}^2    + \int_0^T\si\int\rho u_{tt}^2dxdt\le C.
\be\la{cxb25}
   \sup\limits_{0\le t\le T}\si \|\nabla u_t\|_{L^2}^2
    + \int_0^T\si\|\rho^{\frac{1}{2}}u_{tt}\|_{L^2}^2dt
\le C.
  \ee
\end{lemma}
\begin{proof} By (\ref{cxb21}) and Lemma \ref{xle1},
\be \la{cxb26}
\|P_t\|_{L^2}\le
C\|u\|_{L^\infty}\|\nabla P\|_{L^2}+C\|\nabla u\|_{L^2}\le C.
\ee
Differentiating (\ref{cxb21}) leads to
\bnn
\nabla P_t+u\cdot\nabla\nabla P+\nabla u\cdot\nabla P+\ga \nabla P {\rm div}u+\ga P  \nabla{\rm div}u=0.
\enn
Hence, by Lemmas \ref{xle1} and \ref{xle2},
\bn\la{cxb27} \|\nabla P_t\|_{L^2}\le C\|u\|_{L^\infty}\|\nabla^2
P\|_{L^2}+C\|\nabla u\|_{L^3}\|\nabla P\|_{L^6}+C\|\nabla^2
u\|_{L^2}\le C,\en
which together with (\ref{cxb26}) implies
\bn \la{cxb28}\sup_{0\le t\le T}\|P_t\|_{H^1}\le C.
\en
By \eqref{cxb21} again, it is easy to check that $P_{tt}$ satisfies
\be\la{cxb29} P_{tt} + \gamma P_t{\rm div}u +
\gamma P{\rm div}u_t + u_t\cdot\nabla P + u\cdot\nabla P_t = 0.
\ee
Multiplying \eqref{cxb29} by $P_{tt}$ and integrating over $\Omega\times[0,T],$ we deduce from \eqref{cxb28}, Lemmas \ref{xle1} and \ref{xle2} that
\bnn\ba
&\int_0^T\|P_{tt}\|_{L^2}^2dt \\
& = -\int_0^T\int\gamma P_{tt}P_t\div udxdt - \int_0^T\int\gamma P_{tt}P\div u_tdxdt  \\
& \quad - \int_0^T\int P_{tt}u_t\cdot\nabla Pdxdt - \int_0^T\int P_{tt}u\cdot\nabla P_tdxdt \\
& \le C\int_0^T\|P_{tt}\|_{L^2}(\|P_{t}\|_{L^3}\|\nabla u\|_{L^6}+\|\nabla u_t\|_{L^2}+\|u_t\|_{L^6}\|\nabla P\|_{L^3}+\|u\|_{L^\infty}\|\nabla P_{t}\|_{L^2})dt\\  & \le C\int_0^T\|P_{tt}\|_{L^2}(1+\|\nabla u_t\|_{L^2})dt\\
& \le \frac{1}{2}\int_0^T\|P_{tt}\|_{L^2}^2dt+C\int_0^T\|\nabla u_t\|_{L^2}^2dt+C\\
& \le \frac{1}{2}\int_0^T\|P_{tt}\|_{L^2}^2dt+C,
\ea\enn
which gives
$\int_0^T\|P_{tt}\|_{L^2}^2dt \le C.$

One can deal with $\n_t$ and
$\n_{tt}$ similarly. Thus (\ref{cxb24}) is proved.

It remains to prove (\ref{cxb25}). First, introducing the function
$$H(t)=(\lambda+2\mu)\int(\div u_t)^{2}dx+\mu\int|\curl u_t|^{2}dx .$$
Since $u_t\cdot n = 0$ on $\partial\Omega$, by Lemma \ref{crle1}, we have
\be\ba\la{cxb40}
\|\nabla u_t\|_{L^2}^2\leq C(\Omega) H(t).
\ea\ee

Differentiating  $(\ref{a1})_2$  with
respect to $t,$ then multiplying by
$u_{tt},$  yield that
\be\la{cxb34} \ba
&\frac{d}{dt}\left((\lambda+2\mu)\int(\div u_t)^{2}dx+\mu\int|\curl u_t|^2dx\right)+2\int\rho|u_{tt}|^2dx  \\
&=\frac{d}{dt}\left(-\int\rho_t\,|u_t|^{2}dx-2\int\rho_t\,u\cdot\nabla u\cdot u_tdx+2\int P_t\,\div u_tdx\right) \\
&\quad +\int\rho_{tt}|u_t|^{2}dx + 2\int(\rho_t\,u\cdot\nabla u)_t\cdot u_tdx-2\int\rho u_t\cdot\nabla u\cdot u_{tt}dx \\
&\quad -2\int\rho u\cdot\nabla u_t\cdot u_{tt}dx - 2\int P_{tt}\,\div u_tdx \\
&\triangleq\frac{d}{dt}I_0 + \sum\limits_{i=1}^5I_i .
\ea \ee
Now we have to estimate $I_i$, $i=0,1,\cdots, 5.$

It follows from $(\ref{a1})_1,$  (\ref{cxb2}), (\ref{cxb3}), (\ref{cxb4}), (\ref{cxb17}) and
(\ref{cxb24}) that
\be \ba \la{cxb35}|
I_0|& =\left|-\int_{
}\rho_t |u_t|^2 dx- 2\int_{ }\rho_t u\cdot\nabla u\cdot u_tdx+
2\int_{ }P_t {\rm div}u_tdx\right|\\
&\le \left|\int_{ } {\rm
div}(\rho u)|u_t|^2dx\right|+C\norm[L^3]{\rho_t}\|u\|_{L^{\infty}}\|\nabla u\|_{L^2}
\norm[L^6]{u_t}+C\|\nabla u_t\|_{L^2}\\
&\le C \int_{} \n |u||u_t||\nabla u_t| dx +C\|\nabla u_t\|_{L^2} \\
&\le C\|u\|_{L^6}\|\n^{1/2} u_t\|_{L^2}^{1/2}\|u_t\|_{L^6}^{1/2}\|\nabla
u_t\|_{L^2} +C\|\nabla u_t\|_{L^2}\\
&\le C\|\nabla u\|_{L^2}\|\n^{1/2} u_t\|_{L^2}^{1/2}\|\nabla u_t\|_{L^2}^{3/2}+C\|\nabla u_t\|_{L^2}\\
&\le \de\|\nabla u_t\|_{L^2}^2+C(\de),\ea\ee
\be \la{cxb36}\ba
|I_1|&=\left|\int_{ }\rho_{tt} |u_t|^2 dx\right|=\left|\int_{ }\div(\rho u)_t|u_t|^2 dx\right|\\
& = 2\left|\int_{ }(\rho_tu + \rho u_t)\cdot\nabla u_t\cdot u_tdx\right|\\
& \le  C\left(\norm[H^1]{\rho_t}\norm[H^2]{u}
  +\|\rho^{{1/2}}u_t\|_{L^2}^{\frac{1}{2}}\|\nabla u_t\|_{L^2}^{\frac{1}{2}}\right)\|\nabla u_t\|_{L^2}^2 \\
& \le C\|\nabla u_t\|_{L^2}^4 + C\\
& \le C\|\nabla u_t\|_{L^2}^2H(t) + C,
\ea \ee
\be \la{cxb37}\ba
|I_2|&=2\left|\int_{ }\left(\rho_t u\cdot\nabla u \right)_t\cdot u_{t}dx\right|\\
& = 2\left|  \int_{ }\left(\rho_{tt}\, u\cdot\nabla u\cdot u_t +\rho_t
u_t\cdot\nabla u\cdot u_t+\rho_t u\cdot\nabla u_t\cdot
u_t\right)dx\right|\\
&\le\norm[L^2]{\rho_{tt}}\|u\|_{L^\infty}\norm[L^3]{\nabla u}\norm[L^6]{u_t}+\norm[L^2]{\rho_t}\|u_t\|_{L^6}^2\norm[L^6]{\nabla u} \\
&\quad+\norm[L^3]{\rho_t}\norm[L^{\infty}]{u}\norm[L^2]{\nabla u_t}\norm[L^6]{u_t}\\
& \le C\norm[L^2]{\rho_{tt}}^2 + C\norm[L^2]{\nabla u_t}^2,\ea \ee
\be\ba\la{cxb38}
|I_3|+|I_4|&= 2\left| \int_{ }\rho u_t\cdot\nabla
u\cdot u_{tt} dx\right| +2\left| \int_{ }\rho u\cdot\nabla u_t\cdot
u_{tt} dx\right|\\& \le   C\|\n^{1/2}u_{tt}\|_{L^2}\left(
\|u_t\|_{L^6}\|\na u\|_{L^3}+\|u\|_{L^\infty}\|\na
u_t\|_{L^2}\right) \\& \le  \de \norm[L^2]{\rho^{{1/2}}u_{tt}}^2 +
C(\de)\norm[L^2]{\nabla u_t}^2 , \ea\ee
and
\be\ba\la{cxb39}
|I_5|&=2\left|\int_{ }P_{tt}{\rm div}u_tdx\right|\\&\le C
\norm[L^2]{P_{tt}}\norm[L^2]{{\rm div}u_t}\\& \le
C(\norm[L^2]{P_{tt}}^2 + \norm[L^2]{\nabla u_t}^2).
\ea\ee
% Due to the regularity of the local solution, (\ref{ba1}), $t\nabla u_t\in C([0,T_*];L^2)$. Thus
%\be \la{sp16}\ba \norm[L^2 ]{\nabla u_t(\cdot,T_*/2)}  &\le \frac{2}{T_*}\norm[L^{\infty}(0,T_*;L^2)]{t\nabla u_t}\\ &\le C, \ea\ee where $C$ may also
Therefore, together with these estimates above, we have
\bnn\la{cxb41} \ba
&\frac{d}{dt}\left(\sigma H(t)-\sigma I_0\right)+\sigma\int\rho|u_{tt}|^{2}dx \\
&\le C(1+\|\nabla u_t\|_{L^2}^2)\sigma H(t)+C(1+\|\nabla u_t\|_{L^2}^2+\|\rho_{tt}\|_{L^2}^2+\|P_{tt}\|_{L^2}^2),
\ea \enn
By Gronwall's inequality, together with \eqref{cxb17}, \eqref{cxb24} and \eqref{cxb35}, and choose $\delta>0,$ such that $C(\Omega)\delta<\frac{1}{2}$, where $C(\Omega)$ is given by \eqref{cxb40}, we get
\bnn\la{cxb42} \ba
&\sup_{0\le t\le T}\sigma H(t)+\int_0^T\sigma\|\rho^{\frac{1}{2}}u_{tt}\|_{L^2}^2dt\le C .
\ea \enn
Hence, by \eqref{cxb40},
\bnn\la{cxb43} \ba
&\sup_{0\le t\le T}\sigma \|\nabla u_t\|_{L^2}^2+\int_0^T\sigma\|\rho^{\frac{1}{2}}u_{tt}\|_{L^2}^2dt\le C.
\ea \enn
\end{proof}
\begin{lemma}\la{xle4}
It holds that for any $q\in(3,6),$ there exists a positive constant $C$ such that
\be\la{cxb44}\ba \sup_{t\in[0,T]}\left(\|\rho- \rho_\infty\|_{W^{2,q}} +\|P-P(\rho_\infty)\|_{W^{2,q}}\right)\le C,\ea \ee
\be\la{cxb45}\ba\sup_{t\in[0,T]} \si \|\nabla u\|_{H^2}^2  +\ia \left(\|\nabla  u\|_{H^2}^2+\|\na^2
u\|^{p_0}_{W^{1,q}}+\si\|\na u_t\|_{H^1}^2\right)dt\le C,\ea \ee
where $p_0=\frac{9q-6}{10q-12}\in(1,\frac{7}{6}).$
\end{lemma}
 \begin{proof}
By \eqref{cxb18} and Lemma \ref{xle1},
\be\la{cxb46} \ba\|\nabla^2 u\|_{H^1} &\le C (\|\n \dot u\|_{H^1}+ \| P-\bar{P}\|_{H^2}+\|u\|_{L^2})\\
 &\le C+C \|\na  u_t\|_{L^2},
 \ea\ee
 where we have use the fact that
  \be \label{cxb47} \ba
  \|\nabla (\n \dot u) \|_{L^2}&\le
 \||\nabla \n |\, |  u_t|  \|_{L^2}+ \|\n \,\nabla   u_t  \|_{L^2}
 + \||\nabla \n|\,| u|\,|\nabla u| \|_{L^2}\\ &\quad
 + \|\n\,|\nabla  u|^2\|_{L^2}
 + \|  \n \,|u |\,| \nabla^2 u| \|_{L^2}\\&\le
 \|\nabla \n \|_{L^3} \|  u_t  \|_{L^6}+ C\| \nabla   u_t  \|_{L^2}
 + C\| \nabla \n\|_{L^3}\| u\|_{L^\infty}\|\nabla u \|_{L^6}\\
 &\quad + C\| \nabla  u\|_{L^3}\| \nabla  u\|_{L^6}
 + C\|     u \|_{L^\infty}\| \nabla^2 u  \|_{L^2}\\ &\le C+C\| \nabla   u_t  \|_{L^2}.
 \ea\ee
 So we deduce from (\ref{cxb46}),
(\ref{cxb3}), (\ref{cxb17}), and (\ref{cxb25}) that
\be\la{cxb48}
\sup\limits_{0\le
t\le T}\si\|\nabla  u\|_{H^2}^2+\ia \|\nabla  u\|_{H^2}^2dt \le
 C.\ee
By Lemma \ref{xle1} and \eqref{cxb24}, we have
\be\la{cxb480}\ba
&\|\na^2u_t\|_{L^2}\\
&\le C(\|(\rho\dot{u})_t\|_{L^2}+\|\nabla P_t\|_{L^2}+\|\nabla u_t\|_{L^2}) \\
&=C(\|\n  u_{tt}+\n_t u_t+\n_t u\cdot\nabla u + \n u_t\cdot\nabla u+\n u\cdot\nabla u_t\|_{L^2}+\|\nabla P_t\|_{L^2}+\|\nabla u_t\|_{L^2})+C\\
&\le C\left(\|\n  u_{tt}\|_{L^2}+ \|\n_t\|_{L^3}\|u_t\|_{L^6}+\|\n_t\|_{L^3}\| u\|_{L^\infty}\|\nabla u\|_{L^6}\right)\\
&\quad+C\left(\| u_t\|_{L^6}\|\nabla u\|_{L^3}+ \| u\|_{L^\infty}\|\nabla u_t\|_{L^2}+\|\nabla P_t\|_{L^2}+\|\nabla u_t\|_{L^2}\right)\\
&\le C\|\n^{\frac{1}{2}}  u_{tt}\|_{L^2} +C\|\nabla  u_t\|_{L^2}+C,
\ea \ee
where in the first inequality, we have utilized a priori estimate similar to \eqref{remark1} since
\be\la{cxb49}\begin{cases}
  \mu\Delta u_t+(\lambda+\mu)\nabla\div u_t=(\rho\dot{u})_t+\nabla P_t \,\,\, &in \,\,\Omega,\\ u_t\cdot n=0\,\,\,and \,\,\,\curl u_t\times n=0\,\,&on \,\,\partial\Omega .
\end{cases}\ee
Hence, due to \eqref{cxb25},
\be\la{cxb50}\ba
\int_0^T\sigma\|\nabla u_t\|_{H^1}^2dt\leq C.
\ea \ee
By Gagliardo-Nirenberg's inequality, \eqref{cxb3}, \eqref{cxb18} and \eqref{cxb25}, we conclude that for any $q\in (3,6)$,
\be\la{cxb53}\ba     \|\na(\n\dot u)\|_{L^q}
&\le C \|\na \n\|_{L^q}(\|\nabla\dot{u}\|_{L^q}+\|\nabla\dot{u}\|_{L^2})+C\|\na\dot u \|_{L^q}\\
&\le C \|\nabla\dot{u}\|_{L^2}+C(\|\na u_t \|_{L^q}+\|\na(u\cdot \na u ) \|_{L^q})\\
&\le C \|\nabla u_t\|_{L^2}+C+C\|\na u_t \|_{L^2}^{\frac{6-q}{2q}}\|\nabla u_t\|_{L^6}^{\frac{3(q-2)}{2q}}\\
&\quad+C(\|u \|_{L^\infty}\|\nabla^{2}u\|_{L^q}+\|\nabla u\|_{H^1}\|\nabla u\|_{H^2})\\
&\le C\sigma^{-\frac{1}{2}}+C\|\nabla u\|_{H^2}+C\sigma^{-\frac{1}{2}}(\sigma\|\nabla u_t\|_{H^1}^2)^{\frac{3(q-2)}{4q}}+C.
\ea \ee
Integrating this inequality over $[0,T],$ by \eqref{cxb2} and \eqref{cxb50}, we obtain
\be\ba\la{cxb55}
\int_0^T\|\nabla(\rho\dot{u})\|_{L^q}^{p_0}dt\leq C .
\ea\ee
On the other hand, by \eqref{remark1}, \eqref{remark2}, \eqref{cxb2} and \eqref{cxb18},
\be \la{cxb52}\ba \|\na^2 u\|_{W^{1,q}}
 &\le C(\|\rho\dot{u}\|_{L^q}+\|\nabla(\rho\dot{u})\|_{L^q}+\|\nabla^{2} P\|_{L^q}+\|\nabla P\|_{L^q}\\
&\quad+\|\nabla u\|_{L^2}+\|P-\bar{P}\|_{L^2}+\|P-\bar{P}\|_{L^q})\\
 &\le C(1 + \|\na  u_t\|_{L^2}+ \| \na(\n\dot u )\|_{L^{q}}+\|\na^2  P\|_{L^{q}}),
 \ea\ee
together with \eqref{cxb21} and \eqref{cxb18}, which yields that
\be\la{cxb51}\ba
(\|\na^2 P\|_{L^q})_t\le& C \|\na u\|_{L^\infty} \|\na^2 P\|_{L^q}   +C  \|\na^2 u\|_{W^{1,q}}   \\\le& C (1+\|\na u\|_{L^\infty} )\|\na^2 P\|_{L^q}+C(1+ \|\na  u_t\|_{L^2})\\&+ C\| \na(\n
\dot u )\|_{L^{q}},
\ea\ee
 Hence, by Gronwall's inequality, it follows from \eqref{cxb3}, \eqref{cxb17}  and \eqref{cxb55} that
\be\ba\la{cxb56}
\sup_{t\in[0,T]}\|\nabla^{2}P\|_{L^q}\leq C ,
\ea\ee
which along with \eqref{cxb17}, \eqref{cxb18}, \eqref{cxb52} and \eqref{cxb55} also gives
\be\ba\la{cxb57}
\sup_{t\in[0,T]}\|P-\bar{P}\|_{W^{2,q}}+\int_0^T\|\nabla^{2}u\|_{W^{1,q}}^{p_0}dt\leq C .
\ea\ee
Similarly, we have
\bnn\sup\limits_{0\le t\le T}\|
\n-\bar{\rho}\|_{W^{2,q}} \le
 C,\enn
 which together with (\ref{cxb57}) leads to (\ref{cxb44}). Thus we finish the proof of Lemma \ref{xle4}.

\end{proof}
\begin{lemma}\la{zxle5}
There exists a constant $C$ such that
\be \la{cxb58}
\sup_{0\le t\le T}\si\left(\|\na u_t\|_{H^1}
 +\|\na u\|_{W^{2,q}}\right)
 +\int_{0}^T\si^2\|\nabla u_{tt}\|_{2}^2dt\le C ,
 \ee
 for any $q\in(3,6)$ .
\end{lemma}
\begin{proof} Differentiating $(\ref{a1})_2$ with respect to $t$ twice
gives \be\la{cxb59}\ba
&\n u_{ttt}+\n u\cdot\na u_{tt}-(\lambda+2\mu)\nabla{\rm div}u_{tt}+\mu\nabla\times\curl u_{tt}\\
&= 2{\rm div}(\n u)u_{tt}
+{\rm div}(\n u)_{t}u_t-2(\n u)_t\cdot\na u_t-(\n_{tt} u+2\n_t u_t)
\cdot\na u\\& \quad- \n u_{tt}\cdot\na u-\na P_{tt}.
 \ea\ee

Then, multiplying (\ref{cxb59}) by $2u_{tt}$ and   integrating over $\Omega ,$ it is easy to check that
\be \la{cxb60}\ba
&\frac{d}{dt}\int_{ }\n
|u_{tt}|^2dx+2(\lambda+2\mu)\int_{ }(\div u_{tt})^2dx+2\mu\int_{ }|\curl u_{tt}|^2dx \\
&=-8\int_{ }  \n u^i_{tt} u\cdot\na
 u^i_{tt} dx-2\int_{ }(\n u)_t\cdot \left[\na (u_t\cdot u_{tt})+2\na
u_t\cdot u_{tt}\right]dx\\&\quad -2\int_{
}(\n_{tt}u+2\n_tu_t)\cdot\na u\cdot u_{tt}dx-2\int_{ }   \n
u_{tt}\cdot\na u\cdot  u_{tt} dx\\&\quad+2\int_{ } P_{tt}{\rm
div}u_{tt}dx\triangleq\sum_{i=1}^5J_i.
\ea\ee

Let us estimate $J_i\,(i=1,\cdots,5)$ one by one.
For $J_1$, by H\"{o}lder's and Young's
inequalities,
\be \la{cxb61} \ba |J_1|&\le
C\|\n^{1/2}u_{tt}\|_{L^2}\|\na u_{tt}\|_{L^2}\| u \|_{L^\infty}\\
&\le \de \|\na u_{tt}\|_{L^2}^2+C(\de)\|\n^{1/2}u_{tt}\|^2_{L^2} .
\ea\ee
It follows from (\ref{cxb17}), (\ref{cxb24}), (\ref{cxb25}) and
(\ref{cxb2}) that
\be \la{cxb62}\ba
|J_2|&\le C\left(\|\n
u_t\|_{L^3}+\|\n_t u\|_{L^3}\right)\left(\| u_{tt}\|_{L^6}\| \na
u_t\|_{L^2}+\| \na u_{tt}\|_{L^2}\| u_t\|_{L^6}\right)\\&\le
C\left(\|\n^{1/2} u_t\|^{1/2}_{L^2}\|u_t\|^{1/2}_{L^6}+\|\n_t
\|_{L^6}\| u\|_{L^6}\right)  \| \na u_{tt}\|_{L^2}  \| \na u_{t}\|_{L^2} \\ &\le \de
\|\na u_{tt}\|_{L^2}^2+C(\de)\si^{-3/2},\ea\ee

\be  \la{cxb63}\ba |J_3|&\le C\left(\|\n_{tt}\|_{L^2}
\|u\|_{L^\infty}\|\na u\|_{L^3}+\|\n_{
t}\|_{L^6}\|u_{t}\|_{L^6}\|\na u \|_{L^2}\right)\|u_{tt}\|_{L^6} \\
&\le \de \|\na u_{tt}\|_{L^2}^2+C(\de)\|\n_{tt}\|_{L^2}^2+C(\de)\si^{-1},
\ea\ee
and
\be  \la{cxb64}\ba
|J_4|+|J_5|&\le C\|\n u_{tt}\|_{L^2} \|\na
u\|_{L^3}\|u_{tt}\|_{L^6} +C \|P_{tt}\|_{L^2}\|\na
u_{tt}\|_{L^2}\\
&\le \de \|\na u_{tt}\|_{L^2}^2+C(\de)\|\n^{1/2}u_{tt}\|^2_{L^2}
+C(\de)\|P_{tt}\|^2_{L^2}.
\ea\ee
Substituting these estimates of $J_i(i=1,\cdots,5)$ into (\ref{cxb60}), using the fact that
\be  \la{cxb65}\ba
\|\nabla u_{tt}\|_{L^2}\leq C(\|\div u_{tt}\|_{L^2}+\|\curl u_{tt}\|_{L^2}) ,
\ea\ee
since $u_{tt}\cdot n=0$ on $\partial\Omega,$ and then choosing $\de$ small enough, we get
\be  \la{cxb66}\ba
&\frac{d}{dt}\|\n^{1/2}u_{tt}\|^2_{L^2}+\|\na u_{tt}\|_{L^2}^2\\
&\le C \left(\|\n^{1/2}u_{tt}\|^2_{L^2}+\|\n_{tt}\|^2_{L^2}+\|P_{tt}\|^2_{L^2}\right)+C \si^{-3/2},
 \ea\ee
which, together with  (\ref{cxb24}), (\ref{cxb25}) and by Gronwall's inequality shows that
\be  \la{cxb67}\ba
\sup_{0\le t\le T}\si\|\n^{1/2}u_{tt}\|_{L^2}^2+\int_{0}^T\si^2\|\nabla u_{tt}\|_{L^2}^2dt\le C.
\ea\ee
Furthermore, it follows from \eqref{cxb480} and \eqref{cxb25} that

\be  \la{cxb68}\ba
\sup_{0\le t\le T}\si\|\nabla u_t\|_{H^1}^2\le C.
\ea\ee

Finally, by \eqref{cxb52}, \eqref{cxb53}, \eqref{cxb25}, \eqref{cxb44}, \eqref{cxb45}, \eqref{cxb67} and \eqref{cxb68}, we find that
\be  \la{cxb69}\ba
\si\|\na^2 u\|_{W^{1,q}}
& \le C\left(\si +\si\|\na  u_t\|_{L^2}+\si\| \na(\n\dot u )\|_{L^{q}}+\si\|\na^2  P\|_{L^{q}}\right)\\
& \le C(\sigma+\sigma^{\frac{1}{2}} +  \si\|\na u\|_{H^2}+\sigma^{\frac{1}{2}}(\sigma\|\na u_t\|_{H^1}^2)^{\frac{3(q-2)}{4q}})\\
&\le C\sigma^{\frac{1}{2}}+C\sigma^{\frac{1}{2}}(\sigma^{-1})^{\frac{3(q-2)}{4q}}\\
&\le C ,
\ea\ee
together with (\ref{cxb67}) and (\ref{cxb68}) yields (\ref{cxb58}) and completes the proof .
\end{proof}

\section{\la{se6}Proofs of  Theorems  \ref{th1} and \ref{th2}}

With all the a priori estimates in Section \ref{se3} and Section \ref{se5} at hand, we will  prove the main results of this paper in this section.

{\it \bf{Proof of Theorem \ref{th1}.}} By Lemma \ref{loc1}, the system (\ref{a1})-(\ref{ch2}) has a unique classical solution $(\rho,u)$ on $\Omega\times
(0,T_*]$ for some $T_*$. Now we will extend the classical
solution $(\rho,u)$ globally in time.

First, by the definition of $A_1(T)$, $A_2(T)$ (see \eqref{As1}, \eqref{As2}), the assumption of the initial data \eqref{dt2} and \eqref{bcbh2}, we have
$$ A_1(0)+A_2(0)=0, \,\, 0\leq\rho_0\leq \bar{\rho},\,\, A_3(0)\leq C_0^{\frac{1}{4}}.$$
Hence, there exists a
$T_1\in(0,T_*]$ such that
\be\la{dlbh1}\ba
0\leq\rho_0\leq2\bar{\rho},\,\,A_1(T_1)+A_2(T_1)\leq 2C_0^{\frac{1}{2}}, \,\, A_3(\sigma(T_1))\leq 2C_0^{\frac{1}{4}}.
\ea\ee

Next, set
\bn \la{dlbh2}
T^*=\sup\{T\,|\,{\rm (\ref{dlbh1}) \ holds}\}.
\en
Naturally, $T^*\geq T_1>0$. On the other hand, for any $0<\tau<T\leq T^*$, one deduces from Lemmas \ref{xle3}-\ref{xle5}
that
 \be \la{dlbh3}\begin{cases}
   \rho-\bar{\rho} \in C([0,T]; W^{2,q}), \\ \na u_t \in C([\tau ,T]; L^q),\quad
 \na u,\na^2u \in C\left([\tau ,T];
 C (\bar{\Omega})\right),\end{cases}\ee
 where one has taken advantage of  the standard
embedding
$$L^\infty(\tau ,T;H^1)\cap H^1(\tau ,T;H^{-1})\hookrightarrow
C\left([\tau ,T];L^q\right),\quad\mbox{ for any } q\in [2,6).  $$
By (\ref{cxb17}), (\ref{cxb25}), (\ref{cxb58}) and $(\ref{a1})_1$,
one gets
\be\ba
&\int_{\tau}^T \left|\left(\int\n|u_t|^2dx\right)_t\right|dt\no
&\le\int_{\tau}^T\left(\|  \n_t  |u_t|^2 \|_{L^1}+2\|  \n  u_t\cdot u_{tt} \|_{L^1}\right)dt\\
&\le C\int_{\tau}^T \left( \| \n|\div u||u_t|^2 \|_{L^2}+\|  |u||\na \n| |u_t|^2 \|_{L^1}+ \|\rho^{\frac{1}{2}}  u_t
\|_{L^2}\|\rho^{\frac{1}{2}}u_{tt} \|_{L^2}\right)dt\\
&\le C\int_{\tau}^T\left( \| \n^{\frac{1}{2}} |u_t|^2 \|_{L^2}\|\na u\|_{L^\infty}+\|  u\|_{L^6}\|\na\n\|_{L^2} \|u_t  \|^2_{L^6}+  \|\rho^{\frac{1}{2}}u_{tt} \|_{L^2}\right)dt\\
&\le C,\ea\ee
which together with \eqref{dlbh3} indicates that
\be\la{dlbh4} \n^{1/2}u_t, \quad\n^{1/2}\dot u \in C([\tau,T];L^2).\ee
Finally, we will claim that \be \la{dlbh5}T^*=\infty.\ee Otherwise,
$T^*<\infty$. Now by Proposition \ref{pr1}, it holds that
\be\la{dlbh6}\ba
0\leq\rho\leq\frac{7}{4}\bar{\rho} ,\,\,\,A_1(T^*)+A_2(T^*)\leq C_0^{\frac{1}{2}},\,\,\, A_3(\sigma(T^*))\leq C_0^{\frac{1}{4}} .
\ea\ee
we deduce from Lemmas \ref{xle4}, \ref{xle5} and
(\ref{dlbh4}) that $(\n(x,T^*),u(x,T^*))$ satisfies
the initial data condition (\ref{dt1})-(\ref{dt3}),
where  $g(x)\triangleq\n^{1/2}\dot u(x, T^*),\,\,x\in \Omega.$ Thus, Lemma
\ref{loc1} asserts that there is a $T^{**}>T^*$ such that
(\ref{dlbh1}) holds for $T=T^{**}$, which contradicts the definition of $ T^*.$
Hence, $T^*=\infty$.

 By Lemmas \ref{loc1} and \ref{xle3}-\ref{xle5}, $(\rho,u)$ is really the unique classical solution defined on $\Omega\times(0,T]$ for any  $0<T<T^*=\infty.$

 It remains to prove \eqref{qa1w}. Multiplying (\ref{lx4}) by $4(P-P(\rho_\infty))^3$, we have
 \be \la{lx55}\ba
& \frac{d}{dt}(\|P-P(\rho_\infty)\|_{L^4}^4)  \\
&=-(4\ga-1)\int(P-P(\rho_\infty))^4\div udx-\ga \int P(\rho_\infty)(P-P(\rho_\infty))^3\div udx\\
&\le C\|P-P(\rho_\infty)\|_{L^4}^4+C\|\nabla u\|_{L^2}^2,
\ea\ee
by\eqref{a16} and \eqref{lx61}, yields
\bnn\ba
& \int_1^\infty|\frac{d}{dt}\|P-P(\rho_\infty)\|_{L^4}^4|dt\leq C  \\
\ea\enn
together with \eqref{lx61}, we have
\bnn\ba
\lim_{t\rightarrow\infty}\|P-P(\rho_\infty)\|_{L^4}^4=0.
\ea\enn
Hence, for all $r\in(2,\infty)$ if $\rho_\infty>0$ and $r\in(\gamma,\infty)$ if $\rho_\infty=0$, we get
\be\la{lx58}\ba
\lim_{t\rightarrow\infty}\|\rho-\rho_\infty\|_{L^r}=0.
\ea\ee
Set $$\phi(t)\triangleq (\lambda+2\mu)\|\div u \|_{L^{2}}^{2}+\mu\|\curl u\|_{L^{2}}^{2}.$$ By \eqref{a16} and \eqref{lx1},
\bnn\ba
& \int_1^\infty|\phi(t)|^2dt\leq C.  \\
\ea\enn
Reviewing our derivation of \eqref{I4}, one can find that  the first term on the left side can be signed with absolute value. Taking $m=0$, we get
\be\la{lx56}\ba
|\phi'(t)|\leq C(\|\rho^\frac{1}{2}\dot{u}\|_{L^2}^2+\|\nabla u\|_{L^2}\|\nabla u\|_{L^4}^2+\|\nabla \dot{u}\|_{L^2}^2),
\ea\ee
which, along with \eqref{a16}, \eqref{lx1}  and \eqref{lx7}, gives
\bnn\ba
& \int_1^\infty|\phi'(t)|^2dt\leq C.  \\
\ea\enn
As a result,
 \be\la{lx57}\ba
\lim_{t\rightarrow\infty}\|\nabla u\|_{L^2}=0.
\ea\ee
Finally, due to
 \bnn\ba
\int\rho^\frac{1}{2}|u|^4dx\leq C\|\rho^\frac{1}{2}u\|_{L^2} \|u\|_{L^6}^3\leq C \|\nabla u\|_{L^2}^3,
\ea\enn
and \eqref{lx57}, we obtain
\bnn\ba
\lim_{t\rightarrow\infty}\|\rho^\frac{1}{8}u\|_{L^4}=0.
\ea\enn
 The proof of Theorem \ref{th1} is completed.

{\it\bf{ Proof of Theorem \ref{th2}.}} For $T>0$, the Lagrangian coordinates of the system are given by
  \be \la{c61}  \begin{cases}\frac{\partial}{\partial \tau}X(\tau; t,x) =u(X(\tau; t,x),\tau),\,\,\,\, 0\leq \tau\leq T\\
 X(t;t,x)=x, \,\,\,\, 0\leq t\leq T,\,x\in\bar{\Omega}.\end{cases}\ee
 By \eqref{dt6}, the transformation \eqref{c61} is well-defined. In addition, by $\eqref{a1}_1$, we find that
 \be\la{c62}\ba
\rho(x,t)=\rho_0(X(0; t, x)) \exp \{-\int_0^t\div u(X(\tau;t, x),\tau)d\tau\}.
\ea \ee
 If there exists some point $x_0\in \Omega$ such that $\n_0(x_0)=0,$ then there is a point $x_0(t)\in \bar{\Omega}$ such that $X(0; t, x_0(t))=x_0$. Hence, by \eqref{c62}, $\rho(x_0(t),t)\equiv 0$ for any $t\geq 0.$

 Now we will prove Theorem \ref{th2} by contradiction. Suppose there exist some positive constant $C_1$ and a subsequence ${t_{n_j}},$ $t_{n_j}\rightarrow \infty$ as $j\rightarrow \infty$ such that $\|\na\n (\cdot,t_{n_j})\|_{L^r}<C_1$. Consequently, by  Gagliardo-Nirenberg's inequality, we get that for $ r\in  (3,\infty)$ and $\theta_1=\frac{r-3}{2r-3}$,
\be\la{c63}\ba\rho_\infty\leq\|\rho(\cdot,t_{n_j})-\rho_\infty\|_{C\left(\ol{\O }\right)} \le C
\|\rho(\cdot,t_{n_j})-\rho_\infty\|_{L^3}^{\theta_1}\|\na \rho(\cdot,t_{n_j})\|_{L^r}^{1-\theta_1},
\ea\ee
which is in contradiction with \eqref{lx58}. So we complete the proof.

\section*{Acknowledgements}
The research of \textsc{J. Li} was
partially supported  by the National Center for Mathematics and Interdisciplinary Sciences, CAS, and  NNSFC Grant (Nos. 11688101, 11525106, and 12071200), and Double-Thousand Plan of Jiangxi Province(No. jxsq2019101008).    The research of \textsc{B. L\"u} was partially   supported by NNSFC (No. 11971217) and Natural Science Foundation of Jiangxi Province (Nos. 20161BAB211002 and and 20212BCJ23027).

\end{document}